\documentclass[aap]{imsart}

\RequirePackage{amsthm,amsmath,amsfonts,amssymb}
\RequirePackage[numbers]{natbib}
\RequirePackage[colorlinks,citecolor=blue,urlcolor=blue]{hyperref}

\startlocaldefs
\theoremstyle{plain}
\newtheorem{theorem}{Theorem}
\newtheorem{lemma}[theorem]{Lemma}
\newtheorem{corollary}[theorem]{Corollary}
\newtheorem{proposition}[theorem]{Proposition}
\newtheorem{conjecture}[theorem]{Conjecture}

\newtheorem*{claim*}{Claim}
\theoremstyle{remark}
\newtheorem{remark}[theorem]{Remark}

\newcommand{\R}{\mathbb{R}}

\newcommand{\Z}{\mathbb{Z}}

\newcommand{\E}{\mathbb{E}}
\newcommand{\var}{\mathbb{V}\mathrm{ar}}
\newcommand{\Var}{\var}
\newcommand{\Cov}{\mathbb{C}ov}
\newcommand{\cov}{\Cov}
\newcommand{\Prob}{\mathbb{P}}
\newcommand{\ind}[1]{\mathbf{1}_{#1}}

\newcommand{\Geom}{\operatorname{Geometric}}
\newcommand{\meet}{\operatorname{meet}}

\newcommand{\abso}{\operatorname{abs}}

\newcommand{\hit}{\operatorname{hit}}
\newcommand{\NVM}{\operatorname{NVM}}
\newcommand{\RW}{\operatorname{RW}}
\newcommand{\indic}[1]{\mathbf{1}_{\left\{#1\right\}}}
\newcommand{\sd}{\Gamma}

\usepackage[draft]{todonotes}

\endlocaldefs

\begin{document}

\begin{frontmatter}
\title{Asymptotic behaviour of the noisy voter model density process}
\runtitle{Asymptotic behaviour of the noisy voter model density process}

\begin{aug}
\author[A]{\fnms{Richard} \snm{Pymar}\ead[label=e1]{r.pymar@bbk.ac.uk}}
\and
\author[B]{\fnms{Nicol\'as} \snm{Rivera}\ead[label=e2]{n.a.rivera.aburto@gmail.com}}
\address[A]{School of Computing and Mathematical Sciences, Birkbeck, University of London, London, WC1E
7HX, UK, \printead{e1}}

\address[B]{Instituto de Estad\'istica, Facultad de Ciencias, Universidad de Valpara\'iso, Valpara\'iso, 2360102, Chile, \printead{e2}}
\end{aug}

\begin{abstract}
Given a transition matrix $P$ indexed by a finite set  $V$ of vertices, the voter model is a discrete-time Markov chain in $\{0,1\}^V$ where at each time-step a randomly chosen vertex $x$ imitates the opinion of vertex $y$ with probability $P(x,y)$. The noisy voter model is a variation of the voter model in which vertices may  change their opinions by the action of an external noise. The strength of this noise is measured by an extra parameter $p \in [0,1]$. 

In this work we analyse the density process, defined as the stationary mass of vertices with opinion 1, i.e.\! $S_t = \sum_{x\in V} \pi(x)\xi_t(x)$, where $\pi$ is the stationary distribution of $P$, and $\xi_t(x)$ is the opinion of vertex $x$ at time $t$.
We investigate the asymptotic behaviour of $S_t$ when $t$ tends to infinity for different values of the noise parameter $p$. In particular, by allowing $P$ and $p$ to be functions of the size $|V|$, we show that, under appropriate conditions and small enough $p$ a normalised version of $S_t$ converges to a Gaussian random variable, while for large enough $p$, $S_t$ converges to a Bernoulli random variable.  We provide further analysis of the noisy voter model on a variety of specific graphs including the complete graph, cycle, torus and hypercube, where we identify the critical rate $p$ (depending on the size $|V|$) that separates these two asymptotic behaviours.
\end{abstract}



\end{frontmatter}

\section{Introduction}
\subsection{Background}
\color{black}
The voter model is a classical interacting particle system, first described independently in \cite{clifford} and \cite{holley}. It has found applications across the sciences including statistical physics \cite{spin}, social sciences \cite{social}, chemistry \cite{catalysts}, ecology \cite{eco}, and sociophysics \cite{socio}. The process can be described simply: a set of vertices (or sites) $V$ are initially each endowed with an opinion (0 or 1) and vertices update their opinions over time. At the update time of a vertex $x$,  an imitation step occurs: $x$ changes its opinion to that of a randomly chosen vertex (according to some transition matrix $P(x,\cdot)$, usually the transition matrix of the simple random walk on a graph $G$). The voter model can evolve in discrete or continuous time: in discrete time at each time step one vertex is chosen (uniformly) to update its opinion, whereas in the continuous time setting each vertex updates at the ring times of an independent exponential clock. After its conception, the voter model attracted the attention of probabilists (and later theoretical computer scientists) and several results have been proved in both the finite-vertex set and infinite-vertex set regimes. In this paper we only consider the finite-vertex setting, where the focus has predominantly been on estimating the consensus time defined as the time at which all vertices have the same opinion \cite{berenbrink2016bounds, cooper2013coalescing, cooper2010multiple, hassin2001distributed, kanade2019coalescence,oliveira2019random}. 

Variations of the voter model are ubiquitous, and have been studied to investigate robustness to perturbations in the dynamics, or to allow for increased applicability. These variations include models with multiple opinions \cite{starnini2012ordering}, the existence of one or more zealots/stubborn vertices \cite{huo2019zealot}, constrained voter models \cite{lanchier}, the antivoter model \cite{Rinott1997}, and the process on dynamic graphs \cite{dynamic} or random graphs \cite{random}. 

In this work we focus on analysing the noisy voter model introduced in \cite{granovsky1995noisy} (and independently in \cite{kirman}), which is a modification of the voter model in which the opinions are affected by an external noise. The dynamics are very similar to the voter model: upon selecting a vertex $x$ for updating, with some probability $p$, $x$ independently re-randomizes its opinion (choosing either opinion 0 or 1 with equal probability), otherwise (with probability $1-p$) $x$ performs an imitation step. Note that by choosing $p=0$ we recover the standard voter model, whereas with $p=1$ only re-randomization occurs.  
The noise allows for more interesting large-time behaviour in finite systems, as configurations of complete consensus are no longer absorbing. The external noise causes spontaneous opinion changes and provides an opposing force to the standard voter dynamics which seeks consensus.

To the best of our knowledge, for  finite vertex set $V$, the probability literature has so far only considered mixing times. Suppose $|V|=n$ and let $P$ be a transition matrix. If $p=1$ (only re-randomization occurs) the process is equivalent to a lazy Markov chain on a hypercube on $2^n$ vertices, and so the mixing time is $\Theta(n\log n)$. When the noise parameter $p\in(0,1)$ is constant and $P$ is the transition matrix of a random walk on a graph, \cite{ramadas2014mixing} shows that in the continuous time setting (where imitation steps occur at rate 1 and re-randomization at constant rate) the mixing time of the  noisy voter model is $O(\log n)$.
This translates to a $O(n \log n)$ mixing time for the discrete case. Recently  Cox, Peres and Steif \cite{cox2016cutoff} proved that, under appropriate conditions, the mixing time of the noisy voter model (in discrete time\footnote{their main results are presented in continuous time, but for comparison purposes we present their result in discrete time}) is $n\log n/(2p)$, and that it exhibits total variation cut-off. The noisy voter model has been studied more extensively in the physics literature. Of particular relevance here is the observation  -- obtained through numerical analyses -- of the possible existence of a critical rate for the noise parameter such that if $p=p(n)$ tends to 0 slower than the critical rate then the noise term dominates the dynamics, whereas if $p$ tends to 0 faster than this rate the dynamics is closer to the voter model \cite{carro2016noisy, peralta2018analytical}. We provide theoretical results demonstrating this phenomenon and conjecture its precise nature for general settings (Conjecture~\ref{conjec}).

Our focus is on a particular function of the noisy voter model  known as the \emph{density process} in the probability literature \cite{chen,cox} and the \emph{link magnetisation} \cite{peralta2018stochastic} or the \emph{degree-weighted magnetisation} \cite{Suchecki_2005} in the physics literature.  Consider a transition matrix $P$ on $n$ vertices, let $p \in [0,1]$ and $\xi_t(v)$ denote the opinion of vertex $v\in V$ at time $t$ in the noisy voter model (formal definitions appear in the next section). The time $t$ state of the density process is defined as $S_t := \sum_{x\in V}\pi(x) \xi_t(x)$ where $\pi$ is the stationary distribution of $P$ (assuming it exists). Thus the density process models the evolution of the weight (according to $\pi$) of opinion $1$. Our motivation for choosing the pre-factor $\pi(x)$ is that, for the non-noisy voter model, the quantity $S_t$ is then conserved under the dynamics, that is, $\mathbb{E}[S_{t+1}\mid S_t]=S_t$ almost surely. The quantity $\frac1{n}\sum_x \xi_t(x)$ would be natural in the context of edge-updates (i.e.\! instead of selecting a vertex for opinion-updating, we select an edge) as it is conserved under this dynamic. Further, if we assume that $P$ is reversible and that each vertex has a spin value $+1$ or $-1$ (instead of an opinion), then the magnetisation of an edge $e$ is the average spin of both endpoints, and the \emph{link magnetisation} is the expected magnetisation of a random edge $(x,y)$ chosen with probability $\pi(x)P(x,y)$. When $P$ is the (lazy) transition matrix of a random walk on a graph, the link magnetisation is just the expected average spin of a randomly chosen edge.

The analysis of the evolution of $S_t$ over time, and especially its large-time behaviour, provides insight into the dependency of the noisy voter model on the noise parameter $p$, allowing us to identify different regimes as $p$ varies in $[0,1]$. In particular, we show a transition from an ordered regime for $p$ small to a disordered regime for $p$ large, which gives support to many of the results empirically found in \cite{peralta2018stochastic}, especially when $P$ is the transition matrix of a simple random walk on a large graph.

Our results are expressed in terms of $S = \lim_{t\to \infty} S_t$ (which exists in distribution). We prove that for $p$ sufficiently  small (depending on the transition matrix $P$), $S$ converges (as $|V|$ tends to infinity) to a Bernoulli distribution (Proposition~\ref{prop:small_p}), indicating that in equilibrium almost all vertices have the same opinion as in the usual voter model. Conversely we show that for $p$ sufficiently large (converging to a non-zero constant or to zero sufficiently slowly) $S$ (normalised appropriately) converges to a Gaussian random variable (Theorem~\ref{thm:normalApprox}, Corollaries~\ref{cor:normalApprox}, \ref{thm:normalApproxGraphs}, and~\ref{cor:gaussianreg}) and so in particular there are positive proportions (according to the stationary distribution $\pi$ of $P$) of each opinion present at equilibrium. As far as we are aware, our work is the first to study the asymptotic behaviour of the noisy voter model in which the noise parameter $p$ can be unbounded from below when the size of the graph tends to infinity. 

Throughout, $P$ is chosen as the transition matrix of the simple random walk on a graph $G$, and our results are expressed in terms of the graph structure. For several graph families we obtain the precise critical rate $p_c=p_c(|V|)$, in the sense that if $p\ll p_c$ then $S$ converges to Bernoulli, and $p\gg p_c$ implies that (normalised) $S$ converges to Gaussian (Proposition~\ref{prop:regularhitn}). In all cases we find that the critical probability $p_c$ is proportional to the inverse of the meeting time of two independent walks (started from stationarity).
Obtaining the critical rate for the cycle and the 2D torus is particularly challenging and requires special treatment and careful analysis (Proposition~\ref{P:grid}). The cycle is particularly interesting as on the cycle the noisy voter model is identical to the stochastic Ising model, see for example \cite{cox2016cutoff} (the relationship between the inverse temperature $\beta$ and the noise parameter $p$ is $4\beta=\log(2/p-1)$.)


There are two key ideas behind the proof of Theorem~\ref{thm:normalApprox} (which is used for all results presented here pertaining to Gaussian convergence): duality and Stein's method. Duality for the voter model has been known since the original works on the model \cite{clifford,cox,holley,liggett} and allows for the application of the rich theory of random walks. Stein's method (introduced in \cite{stein}) is a powerful technique for proving approximation results (see for example the survey \cite{steinsurvey}). An application of a particular version in~\cite{rollin2008note} is a key contribution to the proof of Theorem~\ref{thm:normalApprox}, providing general conditions for Gaussian convergence. 
This approximation has been used in \cite{marinov2013counting} to obtain conditions for Gaussian convergence in a voter-like model. Stein's method has also been used to obtain Gaussian convergence in the antivoter model for the proportion of vertices with opinion 1 at stationarity \cite{Rinott1997}. More generally, versions of Stein's method have been used with other particle system/statistical mechanics problems:  to various systems including the voter model \cite{goldstein2018stein}; to discrete Gibbs measures \cite{eichelsbacher2008stein};  to the magnetisation in the Curie–Weiss model (Ising model on the complete graph) \cite{bresler2019stein,chatterjee2010applications,chatterjee2011nonnormal,eichelsbacher2010stein}.

In the remainder of this section we present the main results.  Then, in Section~\ref{S:dual} we construct a dual process to the noisy voter model. While only a small modification is required to the standard dual (to the voter model), for completeness we include the argument. We then use this duality to prove some key properties of the noisy voter model in Section~\ref{sec:properties} including identities for the mean and variance of $S$ and the proof of all the results of Section~\ref{sec:Results}, except Theorem~\ref{thm:normalApprox} which features in Section~\ref{sec:normalApprox}, and Proposition~\ref{P:grid} (results for the cycle and torus). The second half of the paper, Section~\ref{sec:toruscycle}, is devoted to the cycle and torus which requires ad-hoc arguments utilising structural properties of these graphs. The Appendix collects useful results used throughout.






\subsection{Formal description of the model and object of interest}\label{sec:formal}
Here we present a formal definition of the noisy voter model and the statistic $S$.
Let $P$ be the transition matrix of an irreducible discrete-time random walk on a finite set of vertices $V$  and let $\pi$ denote the unique stationary distribution of $P$. The noisy voter model on $V$ associated with $P$ is a discrete-time Markov chain $(\xi_t)_{t \in\mathbb{N}}$ in the space $\{0,1\}^V$ in which all vertices have an initial opinion in $\{0,1\}$. At each time-step, with probability $p$ select a uniform vertex $x$ and re-randomise its opinion, i.e.\! choose an opinion from $\{0,1\}$ uniformly. Otherwise, (with probability $q = 1-p$) perform a standard (imitation) step of the voter model. 

The transition matrix of $(\xi_t)_{t \in\mathbb{N}}$, denoted $Q$, can be given explicitly. Suppose the configurations $\xi$ and $\xi'$ differ only in the value of $x \in V$, i.e. $\xi(y) \neq \xi'(y)\iff y=x$, then
\begin{align}\label{eqn:defQ}
Q(\xi, \xi') = \Prob(\xi_1 = \xi'| \xi_0 = \xi)  = \left\{
    \begin{array}{ll}
      \frac{p}{2n}+\frac{q}{n}\sum_{y \in V} P(x,y) \xi(y) & \text{if $\xi(x) = 0$,} \\
      \frac{p}{2n}+\frac{q}{n}\sum_{y \in V} P(x,y) (1-\xi(y))  & \text{if $\xi(x) = 1$,}\\
    \end{array}
  \right.
\end{align}
where $n=|V|$.
If $\xi$ and $\xi'$ differ at more than one vertex then $Q(\xi, \xi') = 0$. Finally, $Q(\xi, \xi) = 1- \sum_{\xi' \neq \xi} Q(\xi,\xi')$. We write $\NVM(V,P,p)$ for this process.

Observe that if $p>0$ then all configurations in $\{0,1\}^V$ are reachable due to re-randomization, i.e.\!  $Q$ is irreducible. Moreover, with probability at least $p/2$, no vertex changes its opinion at a given time step, so  $Q$ is aperiodic.  It follows that $\NVM(V,P,p)$ has a unique stationary distribution, which we denote by $\sd$.

Given a vector $\xi \in \{0,1\}^{V}$, we define the object of interest $S(\xi):= \sum_{x \in V} \pi(x)\xi(x)$ where $\pi$ is the stationary distribution of $P$ and $\xi$ is a random $\{0,1\}$-valued vector
 with distribution~$\sd$. Thus $S=S(\xi)$ is a random variable, and the value of $S$
is given by the $\pi$-measure of the set of vertices with opinion 1.

When $p=0$, the noisy voter model reverts to
the voter model with transition matrix $P$. In such a case, the voter model reaches an ordered state of consensus, that is, a fixed state where all opinions are the same. Thus the process $\NVM(V,P,p)$ does not have a unique stationary distribution; indeed, it has two absorbing states: the state where all opinions are 1, and another where all opinions are 0.

On the other hand, if $p=1$, then only re-randomization occurs, and thus the chain is a lazy random walk on the hypercube. It follows that  $\sd(\xi)=1/2^n$ for all $\xi \in \{0,1\}^V$, that is, the components of $\xi$ are i.i.d.\! Bernoulli random variables of parameter $1/2$. 
In this case we have $\E(S)=1/2$ and $\var(S)=\sum_{x\in V}\pi^2(x)/4$, and so $\frac{S-1/2}{\sqrt{\var(S)}}$ converges in distribution (when $n$ tends to infinity) to a standard Gaussian if $\max_{y\in V} \{\pi(y)^2/\sum_{x\in V}\pi(x)^2\}\to 0$ by Lindeberg's central limit theorem.

Our results, presented below, provide much finer detail on the transition from order to disorder.

\subsection{Results}\label{sec:Results}

We introduce some more notation before presenting the main results. Set $\pi^* := \max_{x\in V} \pi(x)$,  $\nu^2 := \sum_{x\in V} \pi(x)^2$, and for $x,y\in V$, $\mu(x,y) := \pi(x)^2P(x,y)/\nu^2$.  
We use the letter $\xi$ to denote a sample from $\sd$, the stationary distribution of  $\NVM(V,P,p)$. 
Finally, let $\sigma^2:= \var(S)$ denote the variance of $S$, and $W:= \frac{S-\E(S)}{\sigma}$ denote the standardisation of $S$. Sometimes, we will consider sequences $\{\NVM(V_n,P_n,p_n):\,n\ge1\}$ of the noisy voter model indexed by $|V_n|=n\geq 1$ and, in this case, we will include the subindex $n$ for all above definitions (e.g.\! $S_n, \sigma_n, W_n, \pi_n$). We will always assume for each $n\ge1$ that $P_n$ is irreducible, and that $p_n$ is non-increasing in $n$. 

We denote convergence in distribution by $\overset{\mathcal D}{\Rightarrow}$, e.g. $X_n \overset{\mathcal D}{\Rightarrow} X$, and all limits given are taken when $n$ tends to infinity (unless otherwise stated).

Our first result gives sufficient conditions for $W_n$ converging to a Gaussian random variable as $n$ tends to infinity. While the statement may appear a little cumbersome, the conditions required follow naturally from a version of Stein's method and the result in this form increases the ease of application. We present simplified versions of the result for particular situations in the ensuing corollaries.

\begin{theorem}[General conditions for Gaussian convergence]\label{thm:normalApprox}
Consider a noisy voter model $\NVM(V,P,p)$ with $|V|=n$, and let $\xi\sim \Gamma$. Let $\Phi$ denote the cumulative distribution function of a standard Gaussian random variable on $\R$. Then there exists a universal constant $C>0$ such that 
\begin{align}\label{eqn:hardcondition1}
\sup_{t\in \R}\Big|\Prob\Big(\frac{S-1/2}{\sigma}\leq t\Big)-\Phi(t)\Big|\leq C\Big\{\Big(\frac{\pi^*}{\sigma}\Big)^3 \frac{n}{p}+ \Big(\frac{\pi^*}{\sigma}\Big)^2\sqrt{\frac{n}{p}}+\frac{\nu^2}{p\sigma^2}\sqrt{\Var\Psi} \Big\},
\end{align}
where the random variable $\Psi$ is given by
$
\Psi:=\sum_{x\in V}\sum_{y\in V}\mu(x,y)\ind{\{\xi(x)\neq \xi(y)\}}.
$
 Moreover, we have that
\begin{align}
\Var(\Psi) \leq 16\sigma^2(\pi^*)^2\nu^{-4}.\label{eqn:boundVarDelta1}
\end{align}
In particular, if a sequence of noisy voter models $\NVM(V_n,P_n,p_n)$ (with $|V_n|=n$) satisfies
\begin{align}\label{eqn:hardcondition2}
\frac {n}{p_n}\left(\frac{\pi^*_n}{\sigma_n} \right)^3 + \frac{\pi^*_n}{\sigma_n p_n} \to  0
\end{align}
then $W_n  \overset{\mathcal D}{\Rightarrow} N(0,1)$.

\end{theorem}

The term $1/2$ in the standardisation of $S$ in equation~\eqref{eqn:hardcondition1} comes from the fact that  $\E(S) =  1/2$, see Proposition~\ref{lemma:expectedS}.
We also remark that \eqref{eqn:hardcondition2} follows from substituting the bound on the variance of $\Psi$ from~\eqref{eqn:boundVarDelta1} into~\eqref{eqn:hardcondition1}.

The following lemma gives two bounds that can be used to obtain immediate applications of  Theorem~\ref{thm:normalApprox}. 

\begin{lemma}[Variance lower bounds]\label{lemma:easyVar1} Consider the noisy voter model
$\NVM(V,P,p$). The following two lower bounds for $\sigma^2$ hold:
\begin{enumerate}
\item $4\sigma^2 \geq \nu^2,$
\item for $p\leq 1/2$,
$4\sigma^2 \geq  \left( 1+4p t_{\hit}\right)^{-1},$ where $t_{\hit} = \max_{x,y\in V}\E_{x}(T_y)$, and $T_y$ is the hitting time of $y \in V$ by the random walk associated with $P$.
\end{enumerate}
\end{lemma}
We present the proof of Lemma~\ref{lemma:easyVar1} in Section~\ref{sec:properties}. Combining Theorem~\ref{thm:normalApprox} with Lemma~\ref{lemma:easyVar1}   yields simple conditions to establish a Gaussian limit regime for a sequence of noisy voter models $\NVM(V_n,P_n,p_n)$, which we present next.

\begin{corollary}[Simple conditions for Gaussian convergence]\label{cor:normalApprox}
For each $n \geq 1$ consider a noisy voter model $\NVM(V_n,P_n,p_n)$ with $|V_n|=n$. If
\begin{align}\label{eqn:easyCondition1}
\left(\frac{\pi^*_n}{\nu_n} \right)^3 \frac n{p_n}  \to0 \quad \text{as }n\to\infty,
\end{align}
then
$
W_n \overset{\mathcal D}{\Rightarrow} \mathcal N(0,1).
$
\end{corollary}
Now consider the case in which $V$ is the vertex set of a connected graph and $P$ the transition matrix of the random walk on this graph. It is standard that the stationary distribution of $P$ is given by $\pi(x) = d(x)/(2m)$ where $d(x)$ is the degree of $x$, and $m$ is the number of edges of the graph. In this setting, we can rewrite the conditions of Corollary~\ref{cor:normalApprox} to obtain the following:

\begin{corollary}[Gaussian convergence for random walks]\label{thm:normalApproxGraphs}
Let $G_n$ be a sequence of connected graphs with vertex sets $V_n$ of size $n$, and let $d_n(x)$ be the degree of vertex $x$ in $G_n$. Consider the noisy voter model $\NVM(V_n, P_n, p_n)$ where $P_n$ is the transition matrix of the random walk on $G_n$. If
\begin{align}\label{eqn:easyConditionGraphs}
\left(\frac{\max_{x \in V} d_n(x)}{\sqrt{\sum_{x \in V} d_n(x)^2}} \right)^{3}\frac{n}{p_n} \to0\quad\text{as }n\to\infty,
\end{align}
then
$W_n \overset{\mathcal D}{\Rightarrow} \mathcal N(0,1).$
\end{corollary}
In the case that $p_n$ is a fixed constant in $(0,1)$, the condition $\pi^\star_n/\nu_n = o(n^{-1/3})$ implies~\eqref{eqn:easyConditionGraphs}, while~\eqref{eqn:easyConditionGraphs} implies $\pi^\star_n/\nu_n \to 0$.
The condition $\pi^\star_n/\nu_n = o(n^{-1/3})$ is satisfied by many families of graphs, e.g.\! for $n$-vertex regular graphs, we have $\pi^\star_n/\nu_n = n^{-1/2}$, in which case we need $p_n \gg n^{-1/2}$ to satisfy equation equation~\eqref{eqn:easyConditionGraphs}. On the other hand
it is not satisfied by, for example, the star graph, since $\nu_n^2/4 \geq 1/16$.

\begin{corollary}[Gaussian convergence for random walks on regular graphs]\label{cor:gaussianreg} Let $G_n$ be a sequence of regular graphs with vertex sets $V_n$ of size $n$, and let $P_n$ be the transition matrix of the random walk on $G_n$, and consider a sequence of noisy voter models $\NVM(V_n,P_n, p_n)$. If $p_n\gg n^{-1/2}$, then $
W_n \overset{\mathcal D}{\Rightarrow} \mathcal N(0,1).$
\end{corollary}

We also show that for very small values of $p_n$, the distribution of $S_n$ converges in distribution to a Bernoulli random variable:

\begin{proposition}[Conditions for Bernoulli convergence]\label{prop:small_p}
For each $n \geq 1$, consider a noisy voter model $\NVM(V_n,P_n,p_n$). Let $M_n$ be the meeting time of two independent continuous-time random walks (moving at rate 1 each according to $P_n$) started from stationarity. If $p_n\E(M_n)\to 0$ then $\sigma_n^2 \to 1/4$ and $S_n \overset{\mathcal D}{\Rightarrow} \mathcal {B}er(1/2)$.
\end{proposition}

The following conjecture is natural; however it seems hard to prove in general, even for regular graphs. 

\begin{conjecture}\label{conjec}
For each $n \geq 1$, consider a noisy voter model $\NVM(V_n,P_n,p_n$) such that $\pi^*_n/\nu_n \to 0$. If $p_n\E(M_n)\to\infty$ then
$
W_n \overset{\mathcal D}{\Rightarrow} \mathcal N(0,1).
$
\end{conjecture}

We expect our techniques are not sufficient to obtain a proof of Conjecture~\ref{conjec}, as the condition in the conjecture involves (expectations of) meeting times of only two particles; on the other hand the conditions appearing in~Theorem \ref{thm:normalApprox} (in particular for the quantity $\var(\Psi)$) involve meeting times of four particles. Thus the conditions in Theorem~\ref{thm:normalApprox}  are likely stronger than needed as probabilities of events involving four particles typically can not be bounded by probabilities of events involving just two particles, except if the graph has nice structural properties such as fast mixing.

Note that Proposition~\ref{prop:small_p} combined with the previous conjecture suggests that for transition matrices $P_n$ with $\pi^*_n/\nu_n \to 0$ the inverse expected meeting time is the critical rate of $p_n$ which separates the ordered and disordered phases.

The next result identifies conditions for which the critical rate $p_c$ is equal to $1/n$. Recall that $t_{\hit} = \max_{x,y\in V}\E_{x}(T_y)$, with $T_y$ is the hitting time of $y \in V$ by the random walk associated with $P$.

\begin{proposition}[General conditions for order-disorder transition]\label{prop:regularhitn}
For each $n \geq 1$, consider a noisy voter model $\NVM(V_n,P_n,p_n)$ such that $t_{\hit}= O(n)$ and $\pi^*_n = O(1/n)$.
\begin{enumerate}
\item If $p_n n\to \infty$ then $W_n\overset{\mathcal D}{\Rightarrow} \mathcal N(0,1)$.
\item If $p_n n \to 0$ then $S_n \overset{\mathcal D}{\Rightarrow} \mathcal Ber(1/2)$.
\end{enumerate}

\end{proposition}

The conditions of Proposition~\ref{prop:regularhitn} are satisfied by $P_n$ the random walk on several graph families including hypercubes, regular expanders, grids and tori in dimension larger than 2, and complete graphs.

On the other hand, these general results do not cover two important cases: the cycle, which we denote $C_n$ to indicate it has $n$ vertices, and the two-dimensional torus, denoted by $\mathbb{T}_n$ to indicate it has $n$ vertices (note that when writing $\mathbb{T}_n$ we are implicitly assuming that $n$ is a square number).  For these graphs, we do not only need a sharp lower bound for $\sigma^2_n$ but also better control of the term at the right-hand side of \eqref{eqn:hardcondition1}, which requires some ad-hoc arguments. For the cycle the proof features several couplings with a random walk on the integers, and thus it cannot be extended to other graphs, whereas our argument on the torus may be extended to other transitive graphs.

\begin{proposition}[Order-disorder transition for cycle and 2D torus]\label{P:grid} Consider the noisy voter model $\NVM(V_n,P_n,p_n)$ where $P_n$ is a  sequence of transition matrices associated with the random walk on the graph $G_n = (V_n,E_n)$, and $p_n$ is non-increasing. Then the following holds true:
\begin{enumerate}
\item\label{thm:grid1} Let $G_n=C_n$ be the cycle on $n$  vertices.  If $p_n n^2\to \infty$ then $W_n \overset{\mathcal D}{\Rightarrow} \mathcal N(0,1)$, and if $p_nn^2\to 0$ then $S_n \overset{\mathcal D}{\Rightarrow} \mathcal Ber(1/2)$.
\item\label{thm:grid2} Let $G_n = \mathbb{T}_n$ be a 2-dimensional torus on $n$ vertices. If $p_nn\log n\to \infty$ then $W_n \overset{\mathcal D}{\Rightarrow} \mathcal N(0,1)$, and if $p_nn\log n\to 0$ then $S_n \overset{\mathcal D}{\Rightarrow} \mathcal Ber(1/2)$.
\end{enumerate}
\end{proposition}
Note that on both the cycle and the torus, the cases of  convergence of $S_n$ to a $\mathcal Ber(1/2)$ random variable follow directly from Proposition~\ref{prop:small_p}.

In order to apply Theorem~\ref{thm:normalApprox} to obtain the Gaussian convergence results in Proposition~\ref{P:grid}, we require bounds on $\sigma^2_n$ and $\Var(\Psi_n)$. Consider the case of the cycle. Using duality, we can relate $\sigma_n^2$ to the probability that two independent random walks on the cycle -- started from uniformly chosen locations -- have meeting time less than an independent Geo$(p)$ random variable (which plays the role of the noise). In fact, $\sigma_n^2$ can be computed explicitly in terms of $p$ and $n$ (for the cycle), see Lemma~\ref{L:cyclevar}. Duality can also be used to relate $\Var(\Psi_n)$ to events involving four coalescing random walks. The relevant events are of the form: two particular walkers meet before any other walkers meet and this first meeting time is less than a Geo$(p)$ random variable which itself is less than the meeting time of another particular pair of walkers. Controlling the probability of such events is challenging due to the complicated dependencies and we make use of several martingale optional stopping arguments together with the FKG inequality. This type of problem is similar to those studied in~\cite{addario2016random}.

%
%
%

\begin{remark}[Discrete-time vs continuous-time]
Although  the results in this paper are given for the discrete-time version of the noisy voter model, they can also be established for a continuous-time version. In the continuous-time setting, we have a transition-rate matrix $R = (r(x,y): x,y \in V)$, and a noise rate $\delta$. Each vertex $x$ adopts the opinion of vertex $y$ at rate $r(x,y)$, and re-randomises its opinion at rate $\delta$ (equivalently, flips its opinion at rate $\delta/2$). We can then transform the continuous-time process into a discrete one with the same stationary distribution by using the standard uniformalisation method. Indeed, for given transition-rate $R$, and noise rate $\delta$, the analogue discrete process has transition matrix $P = I+ \frac{R}{r_{\max}}$, where $r_{\max} = \max\{|r(x,x)|: x \in V\}$ is the maximum transition rate from any $x \in V$, and re-randomisation probability $p = \delta/(r_{\max} + \delta)$. Note the stationary distribution of $P$ and $R$ are the same, so the results of this paper are consistent in both settings.
We can also transform the discrete model into a continuous one; for this, each vertex $x$ re-randomises its opinion at rate $p$, and adopts the opinion of vertex $y$ at rate $(1-p)P(x,y)$. 
\end{remark}
{\sc Notation.}\hspace{1em}
We make use of standard asymptotic notation. Given positive real-valued functions $f$ and $g$, we say that $f(x) = O(g(x))$ or $f(x)\lesssim g(x)$ if $\limsup_{x\to \infty}  f(x)/g(x)\leq C$ for some constant $C>0$. 

Similarly $f(x) = \Omega(g(x))$ or $f(x)\gtrsim g(x)$ if $\liminf_{x\to \infty}  f(x)/g(x)\geq  c>0$ for some constant $c$; and $f(x) = \Theta(g(x))$ or $f(x)\asymp g(x)$ if and only if $f(x) = O(g(x))$ and $f(x) = \Omega(g(x))$.  Finally, $f(x) = o(g(x))$ if $\lim_{x\to \infty} f(x)/g(x) = 0 $, in particular $f(x) = o(1)$ implies that $\lim_{x\to \infty} f(x) = 0$.

We use the following set notation: if $S$ is a finite set and $k\in\mathbb{N}$, we write $S^k$ for the set of $k$-tuples from $S$ and $(S)_k$ for the set of $k$-tuples from $S$ without repeats.  Additionally, we use standard language and notation of basic graph theory, in particular $x\sim y$ means that vertices $x$ and $y$ are neighbours. 

\section{The dual process}~\label{S:dual}
It is classical that the voter model has a dual process of coalescing random walks, see for example \cite{liggett}. The noisy voter model also enjoys a coalescing random walk dual process that allows us to recast probabilities involving the noisy voter model in terms of probabilities involving the dual (see Proposition~\ref{thm:duality}) which are often more tractable, and allows for more informative sampling of the stationary distribution~$\Gamma$.

To be more specific, we use a version of Stein's method (Theorem~\ref{thm:rollin}) to obtain conditions (involving variance and covariance) for Gaussian convergence. The dual process will allow us to transform the variance, respectively covariance, estimation problem into one of estimating probabilities of events related to interactions between two, respectively four, particles.

\subsection{Voter model with stubborn vertices}
Consider a noisy voter model $\NVM(V,P,p)$. Before constructing its dual, we introduce a process called the \emph{voter model with stubborn vertices}, so-called as it features vertices whose opinions remain unchanged over time (see \cite{Yildiz2013Binary} for further discussion on the voter model with stubborn vertices).

The role of the stubborn vertices is to encode each re-randomisation step of the noisy voter model as instead a step of the (non-noisy) voter model. To this end, we augment the vertex set $V$ with two additional (stubborn) vertices denoted $\textbf{0}$ and $\textbf{1}$, set $\widetilde V=V\cup\{\textbf{0},\textbf{1}\}$, and define a transition matrix $\widetilde P$ on $\widetilde V$ as
\begin{equation}\label{TPxy}
\widetilde P(x,y) = \left\{
        \begin{array}{ll}
            qP(x,y) & \quad x,y \in V, \\
            p/2 & \quad x \in V, y \in \{\textbf{0},\textbf{1}\},\\
            1 & \quad y=x, x \in \{\textbf{0},\textbf{1}\},\\
            0 & \quad \text{otherwise}.
        \end{array}
    \right.
\end{equation}
Observe that the state of all 0s and the state of all 1s are the unique absorbing states of $\widetilde P$.

The voter model with stubborn vertices is the discrete-time model in which each vertex in $\widetilde V$ has an opinion in $\{0,1\}$ and which at each step chooses uniformly a vertex $x$ from $V$ (not $\widetilde V$) and updates the opinion of the chosen vertex to the opinion of vertex $y$ with probability $\widetilde P(x,y)$. Stubborn vertex $\textbf{0}$ (respectively $\textbf{1}$) has opinion 0 (respectively 1) for all time.

Writing $(\xi_t)_{t\ge0}$ for this process of opinions (we take its state space to be $\{0,1\}^V$ but can extend $\xi_t$ to a function on $\widetilde V$ by setting $(\xi_t(\textbf{0}),\xi_t(\textbf{1}))=(0,1)$) we see that $(\xi_t)_{t\ge0}$ is precisely the noisy voter model. The purpose of this new process will become apparent in the next section.

\subsection{The dual process}\label{SS:dual}


Consider the following discrete-time process: initialise by placing a particle on each $x\in V$ and label it with its host vertex. At each time step choose a vertex $x$ from $V$ uniformly and move all particles at $x$ to $y\in\widetilde V$ with probability $\widetilde P(x,y)$ with $\widetilde P$ defined in~\eqref{TPxy}. Note that particles stay together upon meeting.

Observe that when a particle reaches \textbf{0} or \textbf{1} it never moves again, and we say the particle is absorbed. Note that all particles are eventually absorbed. Indeed, consider particle $x$: at each time-step, the vertex containing $x$ is chosen with probability $1/n$, and then it is absorbed into $\{\textbf{0}, \textbf{1}\}$ with probability $p$. Thus, particle $x$ is absorbed in $\Geom(p/n)$ steps, and so all particles are absorbed in finite time almost surely. We define the absorption time $\tau_{\abso}$ as the  time that all particles are absorbed. Formally, denote by $Y_t$ the vector whose entries are the positions of all particles at time $t$, i.e.\! $Y_t(x)$ denotes the  position of particle $x$ at time $t$ (note that $Y_0(x) = x$). The absorption time is
 $$\tau_{\abso} := \min\{t: Y_t(i) \in \{\textbf{0}, \textbf{1}\} \text{ for all $i \in V$}\}.$$
Finally, define the random variable $B(x)$ to be 1 if particle $x$ is absorbed by  $\textbf{1}$, otherwise $B(x) = 0$. For purposes of the analysis, it is convenient to define $B(\textbf{1})$ and $B(\textbf{0})$ as constants taking the values 1 and 0 respectively (i.e. if a particle starts in \textbf{1}, then it is already absorbed by \textbf{1}). 

The relation between the voter model with stubborn vertices $(\xi_t)_{t \geq 0}$, and the coalescing particle system $(Y_t)_{t \geq 0}$ is established in the following result.

\begin{proposition}[Duality]\label{thm:duality}
Let $\sd$ be the stationary distribution of the noisy voter model (or voter model with stubborn vertices). For each $\xi \in \{0,1\}^V$, it holds that
\begin{align}\label{eqn:Bstationary}
\Prob(B(i) = \xi(i), \forall i \in V) = \sd(\xi),
\end{align}
i.e. the distribution of the random vector $(B(i))_{i \in V}$ is the same as the stationary distribution of the noisy voter model (or voter model with stubborn vertices).
\end{proposition}

The proof of Proposition~\ref{thm:duality} is a straightforward application of the graphical representation of the voter model/coalescing random walks; however, for completeness we include a proof.

\begin{proof}

Let $(\xi_t)_{t \geq 0}$ be the voter model with stubborn vertices and $(Y_t)_{t \geq 0}$ the coalescing particle system. Recall that $\xi_t(\mathbf{0}) = 0$ and $\xi_t(\mathbf{1}) = 1$.

For a pair of vertices $(v,w) \in V \times \widetilde V$ define $M_{v,w}: \{0,1\}^{\tilde V} \to \{0,1\}^{\tilde V}$ such that $\xi' =M_{v,w}\xi$ satisfies $\xi'(u) = \xi(u)$ for every $u \in \tilde V \setminus \{v\}$, and 
$$\xi'(v) = 
\begin{cases}
    \xi(w) & \text{ if $w\in V$}\\
    1 &  \text{ if $w = \mathbf{1}$}\\
    0 &  \text{ if $w = \mathbf{0}$}
\end{cases},
$$
in other words, $\xi'$ is the result of vertex $v$ imitating the opinion of vertex $w$ in state $\xi$.

Let $Y$ be a state of the dual process, i.e. $Y \in \widetilde V^V$ where $Y(i)$ indicates the current position of particle $i$. For $(v,w) \in V \times \widetilde V$, define the operator $M'_{v,w}: \widetilde V^V \to  \widetilde V^V$ that moves all particles located in vertex $v$ to vertex $w$. that is, for any particle $i$ such that $Y(i) = v$, then  $Y'=M'_{v,w}Y$ is such that $Y'(i) = w$.

We start by proving that for each $T \geq 0$ and each $\xi, \xi' \in \{0,1\}^V$, it holds that
\begin{align}\label{eqn:dualForm1}
\Prob(\xi_T = \xi'| \xi_0 = \xi) = \Prob(\xi(Y_T) = \xi'),
\end{align}
where $\xi(Y_t)(i)=\xi(Y_t(i))$.
To prove equation~\eqref{eqn:dualForm1} we  couple the voter model and its dual on $t = 0,\ldots, T$. Recall that a coupling $((\xi^c_t, Y^c_t): t \in \{0,\ldots, T\})$ is such that marginally the distributions of $(\xi_t)_{t=0}^T$ and $(\xi^c_t)_{t=0}^T$ are the same and the distributions of $(Y_t)_{t=0}^T$ and $(Y^c_t)_{t=0}^T$ are also the same.


Define a probability measure $m$ on $V \times \widetilde V$ as
\begin{equation}\label{Mvw}
m(v,w) = \frac{\widetilde P(v,w)}{n}.
\end{equation}
Let $\mathbf{u}=((v_t,w_t): t \in \{1,2,\ldots,\})$ be an i.i.d collection of random elements in $V \times \widetilde V$ sampled with distribution $m$. We use $\mathbf{u}$ to couple the processes.

For $t \in \{1,\ldots, T\}$ define $M_t = M_{v_t,w_t}$. Define the $\{0,1\}^V$-valued process $(\xi^c_t: t \in \{0,1,\ldots, T\})$ as
$$ \xi^c_t = \begin{cases}
      \xi, & \text{if }t = 0, \\
      M_t \xi^c_{t-1}, & \text{if } t \in \{1,\ldots, T\} ,
   \end{cases}
   $$
where $\xi \in \{0,1\}^V$ is given. Observe that $(\xi^c_t: t \in \{0,1,\ldots, T\})$ is distributed as the first $T$ rounds $(\xi_t)_{t \geq 0}$ starting with initial opinion $\xi$.

On the other hand, for $t \in \{1,\ldots, T\}$, define $M'_t = M'_{v_{T-t+1}, w_{T-t+1}}$. Define the $\widetilde V^V$-valued process $(Y^c_t: t \in \{0,1,\ldots, T\})$ as
\begin{align}
Y^c_t = \begin{cases}
      \iota, & \text{if }t = 0, \\
      M'_tY^c_{t-1}, & \text{if } t \in \{1,\ldots, T\} ,
   \end{cases}
\end{align}
where $\iota$ is the identity map.
In this case $(Y^c_t: t \in \{0,\ldots, T\})$ is distributed as the first $t$ rounds of the dual coalescing process $(Y_t)_{t \geq 0}$.

The key observation is that for every $i$  and for all $t \in \{0,\ldots, T\}$
\begin{align}\label{eqn:induction1}
\xi^c_{T-t}(Y^c_{t}(i)) = \xi^c_{T}(i), 
\end{align}
 where we assume that $\xi_t(\textbf{0}) = 0$ and $\xi_t(\textbf{1}) = 1$. To prove~\eqref{eqn:induction1} we fix in advance the sequence $\mathbf{u}_T=((v_t,w_t) \in V \times \widetilde V: t \in \{1,\ldots, T\})$ and use induction.

For $t = 0$, since $Y^c_0(i) = i$ we have $\xi^c_{T}(Y^c_{0}(i)) = \xi^c_{T}(i)$. Assume the result holds for some $t <T$. We  prove it holds for $t+1$. Using that $M_{T-t} = M_{v_{T-t}, w_{T-t}}$ from the definition of $\xi^c_{T-t} = M_{T-t} \xi^c_{T-t-1}$, we have that
\begin{align}\label{eqn:dualInd1}
\xi^c_T(i) = \xi^c_{T-t}(Y^c_t(i)) = \begin{cases}
      \xi^c_{T-t-1}(w_{T-t}), & \text{if }Y^c_t(i) = v_{T-t}, \\
     \xi^c_{T-t-1}( Y^c_t(i)) , & \text{if } Y^c_t(i) \neq v_{T-t} .
   \end{cases}
\end{align}
On the other hand, by using that $M'_{t+1} = M'_{v_{T-t}, w_{T-t}}$,
\begin{align}\label{eqn:dualInd2}
Y^c_{t+1}(i) = \begin{cases}
      w_{T-t}, & \text{if }Y^c_t(i) = v_{T-t}, \\
      Y^c_t(i) , & \text{if } Y^c_t(i) \neq v_{T-t} .
   \end{cases}
\end{align}
Therefore, from equations~\eqref{eqn:dualInd1} and~\eqref{eqn:dualInd2} we get that $\xi^c_{T-t-1}(Y^c_{t+1}(i)) = \xi^c_T(i)$ for all $i \in V$, completing the proof of~\eqref{eqn:induction1}.

A particular case of equation~\eqref{eqn:induction1}  for our coupled variables
is that $\xi^c_0(Y^c_T) = \xi^c_T$.
Then, for all $\xi,\xi' \in \{0,1\}^V$ we have
\begin{align}\label{eqn:equalCoupling}
\Prob(\xi_T = \xi'| \xi_0 = \xi) = \Prob(\xi^c_T = \xi'| \xi^c_0 = \xi) = \Prob(\xi(Y^c_T) = \xi') = \Prob(\xi(Y_T) = \xi').
\end{align}
Now that equation~\eqref{eqn:dualForm1} is proved, we can forget about the coupling and work with the original processes. By taking limits in both sides of equation~\eqref{eqn:dualForm1} when $T$ tends to infinity, we can recover equation~\eqref{eqn:Bstationary}. Indeed, on the left-hand side of equation~\eqref{eqn:dualForm1} we have that
\begin{align}\label{eqn:dualForm2}
\lim_{T \to \infty}\Prob(\xi_T = \xi'| \xi_0 = \xi) = \sd(\xi')
\end{align}
since the stationary distribution of $\xi_T$ is $\sd$. On the other hand, since $\lim_{T \to \infty}\Prob(\tau_{\abso}\leq T) = 1$, by taking limit on the right-hand side of equation~\eqref{eqn:dualForm1} we have\belowdisplayskip=-12pt
\begin{align*}
\lim_{T \to \infty} \Prob(\xi(Y_T) = \xi') &= \lim_{T \to \infty}\Prob(\xi(Y_T) = \xi', \tau_{\abso} \leq T)\\
&=\lim_{T \to \infty}\Prob(\xi(Y_{\tau_{\abso}}) = \xi', \tau_{\abso} \leq T)\\
&=\Prob(\xi(Y_{\tau_{\abso}}) = \xi')\\
 &= \Prob(B(i) = \xi'(i), \forall i). 
\end{align*}
\end{proof}

\begin{remark}[Dual time relation]\label{remark:continuoustimedual}
Since the speed of the dual process is not relevant (in the sequel we are only concerned with the orderings of meeting and absorption times), sometimes it will be convenient for us to consider the continuous-time version of the process defined as follows.  Each vertex $x \in V$ is provided with an independent exponential clock of rate $1/(1-p)$. Each time the clock of $x \in V$ rings, it moves all its particles, if any, using the transition matrix $\widetilde P$. Note that with probability $\widetilde P(x,x)$ the particles do not move. Thus, a vertex rings and any present particles jump at rate $1$ (deciding where to move by using $P$), and at rate $p/(1-p)$ any particles on a particular vertex are absorbed into one of the stubborn states.

The continuous-time version is slightly more convenient as particles are independent until they meet or  are absorbed by  one of the stubborn vertices.
 
We finally note that moving between continuous and discrete-time does not affect the validity of Proposition~\ref{thm:duality}.
\end{remark}


\section{Basic properties of $\NVM(V,P,p)$}\label{sec:properties}

In this section we prove some basic properties of $\NVM(V,P,p)$, in particular properties of the random variable $S= \sum_{x \in V} \pi(x)\xi(x)$ where $\xi$ is sampled from the stationary distribution $\sd$ of $\NVM(V,P,p)$, and $\pi$ is the stationary distribution of $P$. These properties feature in the proofs of the results presented in Section~\ref{sec:Results}.
In all cases we assume $p \in (0,1)$ to avoid trivial cases.

Associated with the dual process of coalescing random walks we define the following random variables. Under the standard convention that $\inf\{\varnothing\} = \infty$, we define the absorption time of particle $i$ as $\tau^{Y(i)} = \inf\{t \geq 0: Y_t(i) \in \{\textbf{0}, \textbf{1}\}\}$, i.e. $\tau^{Y(i)}$ is the hitting time of the set $\{\textbf{0}, \textbf{1}\}$. Also, for $i,j \in V$ define $\tau^{Y(i),Y(j)} = \inf\{\tau^{Y(i)}, \tau^{Y(j)}\}$.  We write $M^{Y(i),Y(j)} = \inf\{t \geq 0: Y_t(i) = Y_t(j) \in V \}$, i.e $M^{Y(i),Y(j)}$ is the meeting time of particles $i$ and $j$. Notice that we are only considering meetings that happen in vertices of $V$ and not in  the stubborn vertices $\textbf{0}$ nor $\textbf{1}$. Hence, $M^{Y(i),Y(j)}$ can be infinity, even if both particles are absorbed by the same vertex. For $i,j \in V$ define the event $E_{ij} = \{M^{Y(i),Y(j)}<\tau^{Y(i),Y(j)}\}=\{M^{Y(i),Y(j)}<\infty\}$. For convenience we define $M^{Y(i),Y(i)} = 0$, and thus $\Prob(E_{ii}=1).$

Finally, we will frequently couple the dual process with an independent random walks process. Given an integer $m\geq 1$, we say that $(X_t(1),\ldots, X_t(m))_{t\geq 0}$ is a $\RW(m,P)$ process if all the random walks $X(i)$ are independent, and they move according to the transition matrix $P$.  We will allow time to be continuous or discrete (which shall be made clear when we use the process); in the case of continuous time the walks jump at rate 1, unless stated otherwise. It follows from independence that the particles in $\RW(m,P)$ do not coalesce upon meeting and there is no absorption (in particular they always move in $V$).  Given $\boldsymbol{x}\in V^m$ (or a measure $q\in V^m$), we use $\Prob^{\RW}_{\boldsymbol{x}}$ ($\Prob^{\RW}_{q}$) to denote the measure with respect to $\RW(m,P)$ when the walks start from $\boldsymbol{x}$ (from $q$). Finally, we write $\RW$ or $\RW(m)$ when $m$ and $P$ are clear from the context. We also write $M^{X_i, X_j}$ to denote the meeting time of the walk $X_i$ and $X_j$ in $\RW(m,P)$.

\begin{proposition}\label{lemma:expectedS} In $\NVM(V,P,p)$ we have
$\E(S) = 1/2.$
\end{proposition}
\begin{proof}
By Proposition~\ref{thm:duality}, we know that $(B(i))_{i \in V}$ and $\xi$ have the same distribution, and so the same marginal distributions. In particular, for all $x \in V$, it holds that $\Prob(\xi(x) = 1) = \Prob(B(x) = 1)$. Moreover, note that $\Prob(B(x) = 1) = 1/2$ since in the dual process, each particle is absorbed by either \textbf{0} or \textbf{1} with equal probability. (Equation \eqref{TPxy} gives the probability that a particle at $v$ transitions to $w$  in a given step.)
Therefore
\[ \E(S) = \sum_{x \in V} \pi(x) \E(\xi(x)) = \sum_{x \in V} \pi(x) \Prob(\xi(x) = 1) = 1/2 .\qedhere\]
\end{proof}

\begin{lemma}\label{lemma:sigma2}
 Consider $\NVM(V,P,p)$, and let $\sigma^2 = \var(S)$. Then
\[\sigma^2 = \frac{1}{4}\sum_{x \in V} \sum_{y \in V} \pi(x)\pi(y)\Prob(E_{xy})
\]
\end{lemma}
\begin{proof}
Recall that
$$\sigma^2 = \var\left(\sum_{x \in V} \pi(x)\xi(x)\right) = \sum_{x\in V}\sum_{y\in V} \pi(x)\pi(y)\Cov(\xi(x),\xi(y)).$$
By Proposition~\ref{thm:duality}, the vector $(B(x))_{x \in V}$ has the same  stationary distribution as $\xi$. Thus
$$\sigma^2 = \sum_{x \in V} \sum_{y \in V} \pi(x)\pi(y)\cov(B(x),B(y)).$$
From the definition of covariance,
\begin{align}
\cov(B(x),B(y)) = \E(B(x)B(y))-\E(B(x))\E(B(y)) = \E(B(x)B(y))-(1/4).
\end{align}
Given the event $E_{xy} = \{M^{Y(x),Y(y)} < \tau^{Y(x),Y(y)}\}$, observe that $\E(B(x)B(y)| E_{xy}) = 1/2$: if both particles meet before either one is absorbed,  they  move together and  their final state is the same. They are absorbed by \textbf 1 with probability $1/2$, and by \textbf{0} otherwise. On the other hand, $\E(B(x)B(y)|E_{xy}^c) = 1/4$ since both particles are absorbed at different times and thus their final states are independent given $E_{xy}^c$.
Hence
\[\cov(B(x),B(y)) = (1/2)\Prob(E_{xy}) + (1/4)\Prob(E_{xy}^c)-1/4 = (1/4)\Prob(E_{xy}).\qedhere\]
\end{proof}

\begin{proof}[Proof of Lemma \ref{lemma:easyVar1}]
We begin with the first item. Since $\Prob(E_{xx}) = 1$ for all $x \in V$, then  Lemma~\ref{lemma:sigma2} yields 
\[\sigma^2=  \var(S) \geq (1/4)\sum_{x \in V} \pi(x)^2 \Prob(E_{xx}) = (1/4)\sum_{x \in V} \pi(x)^2 =\nu^2/4.\]
For the second item,  as was discussed previously $\Prob(E_{xy})$ has the same value if we consider discrete or continuous time since this only affects the speed of the process. For this proof it is more convenient to work in continuous time for which particles in the dual move at rate $1$ and are absorbed by the stubborn vertices at rate $\delta:= p/(1-p)$. Note that up to time $M^{Y(x) ,Y(y)} \wedge \tau^{Y(x),Y(y)}$, the particles $Y(x),Y(y)$ move like independent continuous-time random walks, and the time $\tau^{Y(x),Y(y)}$ is independent of $Y(x),Y(y)$ and is exponentially distributed with rate $2\delta$. Hence, we can naturally couple $(Y(x),Y(y))$ with a $\RW(2,P)$ process $(X_1,X_2)$ up to time $M^{Y(x) ,Y(y)} \wedge \tau^{Y(x),Y(y)}$. Thus
\begin{align}\label{eqn:above}1-\Prob(E_{x,y})=\Prob(M^{Y(x),Y(y)}>\tau^{Y(x),Y(y)}) = \Prob^{\RW(2)}_{(x,y)}(M^{X_1,X_2}>Z)\end{align}
where $M^{X_1,X_2}$ denotes the meeting time of the independent random walks $X_1,X_2$ and $Z$ is an independent exponential random variable of rate $2\delta$. 

By multiplying \eqref{eqn:above} by $\pi(x)\pi(y)$ and summing over all $x,y\in V$,  Lemma~\ref{lemma:sigma2} yields
\begin{align}\label{eqn:random948g424s}
    1-4\sigma^2 = \Prob_{\pi^2}^{\RW}(M^{X_1,X_2}>Z) = \int_0^{\infty} \Prob_{\pi^2}^{\RW}(M^{X_1,X_2}>t)2\delta e^{-2\delta t}dt,
\end{align}

Recall that $X_1$ and $X_2$ move on state space $V$ at rate $1$ according to the transition matrix $P$. Then from \cite[Lemma 1.7]{olive}, we have  $$\Prob_{\pi^2}(M^{X_1,X_2}>t)\leq e^{-t/t_{\hit}},$$
where we recall that $t_{\hit} = \max_{x,y\in V} \E_{x}(T_y)$, with $T_y$ is the hitting time of vertex $y\in V$. Therefore,
\begin{align*}
\int_0^{\infty} \Prob_{\pi^2}^{\RW}(M^{X_1,X_2}>t)2\delta e^{-2\delta t}dt &\leq  2\delta\int_0^{\infty} \exp\left(-t\left(\frac{1}{t_{\hit}}+2\delta \right)\right)dt\\& = \frac{2\delta}{\frac{1}{t_{\hit}}+2\delta} = \frac{2\delta t_{\hit}}{1+2\delta t_{\hit}} .
\end{align*}
By plugging this inequality  into~\eqref{eqn:random948g424s} we obtain
\begin{align*}
    4\sigma^2 \geq  1- \frac{2\delta t_{\hit}}{1+2\delta t_{\hit}} = \frac{1}{1+2\delta t_{\hit}}.
\end{align*} 
Note that $t_{\hit}$ is the same in either continuous or discrete time, so the identical bound holds in either setting. Using the previous inequality and the assumption $p\leq 1/2$, we obtain
\[
    4\sigma^2 \geq   \frac{1}{1+2\delta t_{\hit}} \geq \frac{1}{1+4p t_{\hit}}.\qedhere
\]
\end{proof}

 Using a similar argument we next prove Proposition~\ref{prop:small_p}.

\begin{proof}[Proof of Proposition~\ref{prop:small_p}]
We shall show that $\sigma_n^2 \to 1/4$, and so, since the Bernoulli random variable of parameter $1/2$ is the only distribution taking values in $[0,1]$ with variance $1/4$, we will have proven the result.

Following the same argument of the proof of Lemma~\ref{lemma:easyVar1} we just have to prove that $$\Prob_{\pi_n^2}^{\RW(2)}(M^{X_1,X_2}>Z)\to 0,$$ where $(X_1,X_2)$ is a continuous-time $\RW(2,P_n)$ process on $V_n$ moving at rate $1$ according to $P_n$, and $Z$ is an exponential random variable of rate $2p_n/(1-p_n)$.

In the following computations we omit the subscript $n$ to ease notation. Let $t_{\meet} = \E_{\pi^2}^{\RW(2)}(M^{X_1,X_2})$, and $\delta = p/(1-p)$, then

\begin{align}
\notag \Prob_{\pi^2}^{\RW(2)}(M^{X_1,X_2}>Z) &= \int_0^{\infty} \Prob_{\pi^2}(M^{Y(x),Y(y)}>t)2\delta e^{-2\delta t}dt\\
&\leq (1- e^{-2\delta t_{\meet}})+t_{\meet}\int_{t_{\meet}}^{\infty}\frac{2\delta e^{-2\delta t}}{t}dt.\label{eqn:randome1928j9bxcb}
\end{align}
Now,
\begin{align*}
    \int_{t_{\meet}}^{\infty}\frac{2\delta e^{-2\delta t}}{t} dt = \sum_{i=1}^{\infty} \int_{it_{\meet}}^{(i+1)t_{\meet}}\frac{2\delta e^{-2\delta t}}{t}dt \leq 2\delta \sum_{i=1}^{\infty}\frac{e^{-2\delta it_{\meet}}}{i}.
\end{align*}
Observe that $\sum_{i=1}^{\infty}e^{-ai}/i = -\log(1-e^{-a})$, and that $e^{-a} \leq 1/(1+a)$ for $|a|<1$. Thus for all $n$ sufficiently large (so that $\delta t_{\meet} <1$), we obtain
\begin{align*}
    \int_{t_{\meet}}^{\infty}\frac{2\delta e^{-2\delta t}}{t}dt\leq  2\delta \log(1+1/(2\delta t_{\meet})).
\end{align*}
We now use that $e^{-x}\le 1-x+x^2/2$ for $x>0$, to give that for $0<x<1$, $1\le \frac{1-e^{-x}}{x}\frac1{1-x/2}$, i.e.\! for $n$ sufficiently large,
\[
1\le \frac{1-e^{-2\delta t_{\meet}}}{2\delta t_{\meet}}\frac1{1-\delta t_{\meet}}.
\]
We thus obtain 
\begin{align*}
    \int_{t_{\meet}}^{\infty}\frac{2\delta e^{-2\delta t}}{t}dt\leq  \frac{1-e^{-2\delta t_{\meet}}}{ t_{\meet}}\frac1{1-\delta t_{\meet}} \log(1+1/(2\delta t_{\meet})),
\end{align*}
and plugging this into \eqref{eqn:randome1928j9bxcb} yields
\begin{align}
   \Prob_{\pi_n^2}^{\RW(2)}(M^{X_1,X_2}>Z) \leq (1- e^{-2\delta t_{\meet}})\left( 1+\frac{\log\left(1+1/(2\delta t_{\meet})\right)}{1-\delta t_{\meet}} \right).
\end{align}

The result follows since $\lim_{x\to 0} (1-e^{-x})(1+\log(1+1/x)/(1-x/2)) = 0$ and $\delta t_{\meet} \to 0$ by assumption.
\end{proof}
\begin{proof}[Proof of Corollary~\ref{cor:normalApprox}]
We will prove that \eqref{eqn:easyCondition1} implies \eqref{eqn:hardcondition2}. We consider two cases: first, suppose that $\big(\frac{\pi^*_n}{\sigma_n}\big)^2n\ge 1$, then, by using the first item of Lemma~\ref{lemma:easyVar1} in equation \eqref{eqn:hardcondition1} we have
\[\frac{n}{p_n}\left(\frac{\pi^*_n}{\sigma_n} \right)^3 + \frac{\pi^*_n}{\sigma_n p_n} =
\frac{\pi^*_n}{\sigma_n p_n}\left(\Big(\frac{\pi^*_n}{\sigma_n}\Big)^2n+1\right)\le 2\left(\frac{\pi^*_n}{\sigma_n}\right)^3\frac{n}{p_n}\le 2\left(\frac{2\pi^*_n}{\nu_n}\right)^3\frac{n}{p_n}=o(1).
\]

Secondly, consider $\big(\frac{\pi^*_n}{\sigma_n}\big)^2n< 1$, thus
\[\frac {n}{p_n}\left(\frac{\pi^*_n}{\sigma_n} \right)^3 + \frac{\pi^*_n}{\sigma_n p_n}  = 
\frac{\pi^*_n}{\sigma_n p_n}\left(\Big(\frac{\pi^*_n}{\sigma_n}\Big)^2n+1\right)\le 2\frac{\pi^*_n}{\sigma_n p_n}<\frac{2}{\sqrt n p_n}=o(1),
\]
since $\sqrt{n}p_n\to \infty$ which is implicit in  statement~\eqref{eqn:easyCondition1} as $(\pi^*_n/\nu_n)^2\geq n^{-1}$.
\end{proof}
\begin{proof}[Proof of Proposition~\ref{prop:regularhitn}]
Recall that $p_n$ is non-increasing, hence for the first item, we divide into the cases that eventually $p_n\leq 1/2$, and $p_n>1/2$. If $p_n>1/2$ we apply directly Corollary~\ref{cor:normalApprox} since $\nu_n=\Theta(1/\sqrt{n})$ and $\pi^*_n = O(1/n)$.  For the case $p_n\leq 1/2$, by Lemma~\ref{lemma:easyVar1} we have that $\sigma_n = \Omega\left(\frac{1}{np_n} \right)$, which replaced on the left-hand side of equation~\eqref{eqn:hardcondition2} yields
\begin{align*}
O\left(\sqrt{p_n/n}+ 1/\sqrt{np_n}\right) = o(1),
\end{align*} 
and so Theorem~\ref{thm:normalApprox} yields the result.

For the second item, by \cite[Proposition 14.5]{aldousfill2014} we have that $\E(M_n) \leq t_{\hit} = O(n)$ where $M_n$ is the meeting time of two independent continuous-time random walks where the initial state is sampled according to $\pi$. The result follows then from Proposition~\ref{prop:small_p}.
\end{proof}

The next result will be useful to prove the Gaussian approximation of the random variable $W$. Suppose we consider  $\NVM(V,P,p)$ with $|V|=n$ starting from $\xi_0 = \xi\sim \sd$. Denote $\xi'= \xi_1$. Recall the definition $W = (S-1/2)/\sigma$, where $S = \sum_{x\in V} \pi(x)\xi(x)$, and similarly define $W'$ (replacing $\xi$ with $\xi'$). Since $\xi\sim \sd$ we have $\var(S) = \var(S')=\sigma^2$ and $\E(S)=\E(S') = 1/2$.
The following lemma relates $W'$ with $W$.

\begin{lemma}\label{lemma:lambda} 
$\E(W'|W) = \left(1-\frac{p}{n}\right)W.$
\end{lemma}

\begin{proof}
Set $A = \{x \in V: \xi(x) = 1\}$ and let $P(x,A) = \sum_{y \in A} P(x,y)$. Then, for $x\in V$, we have
\begin{align}\nonumber
\E(\xi'(x)|\xi)&= \frac{(n-1)}{n}\xi(x) + \frac{1}{n}\left(\frac{p}{2}+qP(x, A)\right).
\end{align}
To see this consider the following argument. With probability $1/n$, vertex $x$ is chosen to update its opinion. If this happens, then with probability $p$ this update is a re-randomisation and further with probability 1/2 its opinion becomes $1$. Otherwise, with probability $q=1-p$, a standard voting step is performed, and thus $x$ samples a vertex with opinion 1 with probability $P(x,A)$.
By multiplying by $\pi(x)$ and  summing over $V$, we have
\begin{align}\label{eqn:partialESn0'}
\E(S'|\xi) &= \frac{n-1}{n}S+ \frac{p}{2n}  +\frac{q}{n}\sum_{x \in V} \pi(x)P(x,A).
\end{align}
Denote $Q(A,B) = \sum_{x \in A} \pi(x)P(x,B)$. For any irreducible transition matrix $P$, it holds that $Q(A,A^c) = Q(A^c, A)$ \cite[Exercise 7.2]{levin2009markov}. Then $Q(V,A) = Q(A,A)+Q(A^c, A) = Q(A,A)+Q(A,A^c) = Q(A,V) = \pi(A)$. Thus
\begin{align}
\sum_{x \in V} \pi(x)P(x,A)&= \pi(A)
=\sum_{x \in V} \pi(x) \xi(x) = S .\nonumber
\end{align}
Plugging into equation~\eqref{eqn:partialESn0'} we obtain
\begin{align}\label{eqn:partialESn0}
\E(S'|\xi) &= S\left(1-\frac{p}{n}\right) + \frac{p}{2n}.\nonumber
\end{align}
Observe that since $S$ is a function of $\xi$ we have $\E(\E(S'|\xi)|S) = \E(S'|S)$. We conclude
\begin{align}
\E(S'|S) - \frac{1}{2} &=
\left(S-\frac{1}{2}\right)\left(1-\frac{p}{n}\right). \nonumber
\end{align}
Dividing by $\sigma = \sqrt{\Var(S)}$ on both sides completes the proof.
\end{proof}

\section{Gaussian approximation - proof of Theorem~\ref{thm:normalApprox}}\label{sec:normalApprox}
\subsection{Proof of equation~\eqref{eqn:hardcondition1}}\label{sec:proofofhard1}
In this section we prove equation~\eqref{eqn:hardcondition1} of Theorem~\ref{thm:normalApprox}. For that, we combine our previous results with Stein's method, leading not only to the asymptotic normality of sequences of $\NVM$ processes, but also to error rates. For such, we will employ a result from R\"ollin \cite{rollin2008note} (Theorem~\ref{thm:rollin} in the Appendix) that gives an upper bound for the Kolmogorov distance between the random variable $W$ and a standard real-valued Gaussian random variable. R\"ollin simplified the conditions of the so-called `exchangeable pair' method developed in \cite{Rinott1997} which has been used to study the antivoter model.

Recall the definition of $\Psi$ from Theorem~\ref{thm:normalApprox} as \[
\Psi:=\sum_{x\in V}\sum_{y\in V}\mu(x,y)\ind{\{\xi(x)\neq \xi(y)\}}
.\]


\begin{proposition}
Consider a noisy voter model $\NVM(V,P,p)$ where $|V|=n$. Let $\Phi$ be the cumulative distribution function of a standard normal random variable. There exists a universal constant $C>0$ such that
\begin{align}
    \sup_{t\in\mathbb{R}}|\Prob(W\leq t) - \Phi(t)| \leq C\Big\{\Big(\frac{\pi^*}{\sigma}\Big)^3 \frac{n}{p}+ \Big(\frac{\pi^*}{\sigma}\Big)^2\sqrt{\frac{n}{p}}+\frac{\nu^2}{p\sigma^2}\sqrt{\Var(\Psi)} \Big\}.
\end{align}
\end{proposition}
\begin{proof}
Consider the random vector of opinions $\xi$ sampled from $\Gamma$, and let $\xi'$ be the vector $\xi$ after one iteration of the noisy voter model from $\xi$. We consider $S = \sum_{x} \pi(x)\xi(x)$ and $S' = \sum_x \pi(x)\xi'(x)$, and denote by $W$ and $W'$ the corresponding standardised versions. 

Note that $\E(W)= 0$ and $\var(W)=1$, and that $\xi'$ has the same distribution as $\xi$ since $\xi$ is sampled from the stationary distribution (so $W'$ has the same mean and variance as $W$). Moreover, note that $|W'-W|\leq \pi^*/\sigma$ since only one vertex may change its opinion.  Finally, by Lemma~\ref{lemma:lambda} we have $\E(W'|W) = (1-p/n)W$. Then, combining the three previous facts with Theorem~\ref{thm:rollin} yields
\begin{align}\label{eqn:rollin193831}
\sup_{t\in \R} |\Prob(W \leq t)-\Phi(t)| \leq \frac{12n}{p} \sqrt{\var(\E((W'-W)^2|W))}+32\frac{A^3}{\lambda}+6\frac{A^2}{\sqrt{\lambda}},
\end{align}
where $A = \pi^*/\sigma$ and $\lambda = p/n$, and $\Phi$ is the cumulative distribution function of a standard Gaussian random variable; and further
\begin{align}\label{eqn:random873djvsv}
  32\frac{A^3}{\lambda}+6\frac{A^2}{\sqrt{\lambda}} \leq 32\left(\left(\frac{\pi^*}{\sigma}\right)^3 \frac{n}{p}+ \left(\frac{\pi^*}{\sigma}\right)^2\sqrt{\frac{n}{p}}\right).
\end{align}

We now proceed to find an upper bound for $\var(\E((W'-W)^2|W))$. We start by noticing that 
$$\var(\E((W'-W)^2|W))\leq \var(\E((W'-W)^2|\xi)),$$
since $W$ is a function of $\xi$. Next, notice that since at most one vertex changes its opinion in one iteration of the noisy voter model we have
\begin{align*}
(W'-W)^2 = \sum_{x\in V} \left(\frac{\pi(x)}{\sigma}\right)^2 \ind{\{\xi'(x)\neq \xi(x)\}}
\end{align*}
and thus
\begin{align}
\E((W'-W)^2|\xi) = \sum_{x\in V} \left(\frac{\pi(x)}{\sigma}\right)^2 \Prob(\xi'(x)\neq \xi(x)|\xi).\label{eqn:random19sczq1}
\end{align}

Denote by $D_x = D_x(\xi):= \{y \in V:\xi(x) \neq \xi(y)\} \subseteq V$ the (random) set of vertices whose opinion differs from the opinion of $x$, and then we claim that
 \begin{align}
 \Prob(\xi'(x) \neq \xi(x)| \xi ) &= \frac{1}{n}\left(\frac{p}{2} + qP(x,  D_x)\right) = \frac{1}{n}\left(\frac{p}{2} + q\sum_{y\in V} P(x,y)\ind{\{\xi(x)\neq \xi'(x)\}}\right).\label{eqn:probDifferent}
  \end{align}
To see this note first of all that for $\xi'(x)$ to differ from $\xi(x)$, we have to choose vertex $x$ in the round, which has probability $1/n$. After that, with probability $p/2$ we re-randomise the opinion of $x$ to the opposite opinion. Otherwise (with probability $q=1-p$), $x$ chooses a random $y$ vertex using $P(x,\cdot)$ and imitates its opinion, obtaining the opposite opinion if and only if $y \in D_x$. By plugging equation~\eqref{eqn:probDifferent} into equation~\eqref{eqn:random19sczq1} we obtain
\begin{align}
 \E((W'-W)^2|\xi) &= \sum_{x\in V} \left(\frac{\pi(x)}{\sigma}\right)^2 \Prob(\xi'(x)\neq \xi(x)|\xi)\nonumber\\
 & =\frac{1}{n}\sum_{x\in V} \left(\frac{\pi(x)}{\sigma}\right)^2 \left(\frac{p}{2} + q\sum_{y\in V} P(x,y)\ind{\{\xi(x)\neq \xi(y)\}}\right)\nonumber\\
 &= \frac{p}{2n}\sum_{x\in V} \left(\frac{\pi(x)}{\sigma}\right)^2+ \frac{q}{n\sigma^2}\left(\sum_{x\in V}\sum_{y\in V} \pi(x)^2  P(x,y)\ind{\{\xi(x)\neq \xi(y)\}}\right)\nonumber\\
 &=  \frac{p}{2n}\sum_{x\in V} \left(\frac{\pi(x)}{\sigma}\right)^2+ q\frac{\nu^2}{n\sigma^2}\left(\sum_{x\in V}\sum_{y\in V} \mu(x,y)\ind{\{\xi(x)\neq \xi(y)\}}\right)\label{eqn:random11239fj4}\\
 &=\frac{p}{2n}\sum_{x\in V} \left(\frac{\pi(x)}{\sigma}\right)^2+q\frac{\nu^2}{n\sigma^2}\Psi\label{eqn:random39391s},
\end{align}
where in equation~\eqref{eqn:random11239fj4} have used the definition of $\mu$, and in \eqref{eqn:random39391s} the definition of $\Psi$.

Taking variance we have
\begin{align*}
\var(\E((W'-W)^2|\xi)) = (1-p)^2\frac{\nu^4}{n^2\sigma^4}\var(\Psi)\leq \frac{\nu^4}{n^2\sigma^4}\var(\Psi),
\end{align*}
concluding that
\begin{align}\label{eqn:random414t}
    \frac{12n}{p}\sqrt{\var(\E((W'-W)^2|W))} \leq \frac{12n}{p}\frac{\nu^2}{n\sigma^2}\sqrt{\var(\Psi)}.
\end{align}

We complete the proof by combining equation~\eqref{eqn:random414t} with equation~\eqref{eqn:random873djvsv}, and by choosing $C = 32$.
\end{proof}

\subsection{Bounding $\var(\Psi)$}
 In this section we consider the $\NVM(V,P,p)$ process with $|V|=n$, and provide a bound on $\var(\Psi)$.  Recall that in the (discrete-time) dual process $Y_t(x)$ denotes the location of particle initially at vertex $x\in V$
and we denote by $B(x)$, the final location of the particle that started at vertex $x$.

For $u,v,x,y \in V$, we set 
\begin{align}\notag
h(u,v,x,y)&:=\mathbb{P}(B(x)\neq B(u),\,B(y)\neq B(v))-\mathbb{P}(B(x)\neq B(u))\mathbb{P}(B(y)\neq B(v))\\
&=\Cov(\ind{\{B(x)\neq B(u)\}},\ind{\{B(y)\neq B(v)\}})\notag\\
&=\Cov(\ind{\{\xi(x)\neq \xi(u)\}},\ind{\{\xi(y)\neq \xi(v)\}})\label{eq:hasCov}
\end{align}
where the last equality follow from Proposition~\ref{thm:duality} assuming $\xi\sim\Gamma$. With this notation we can write 
\begin{align*}
\var(\Psi) &= \var\left(\sum_{x\in V}\sum_{u\in V}\mu(x,u)\ind{\{\xi(x)\neq \xi(u)\}}\right)\nonumber\\
&=\sum_{x,u, y,v \in V}\mu(x,u)\mu(y,v)\Cov\left(\ind{\{\xi(x)\neq \xi(u)\}},\ind{\{\xi(y)\neq \xi(v)\}}\right)\\
&= \sum_{x,u, y,v \in V}\mu(x,u)\mu(y,v) h(u,v,x,y).
\end{align*}
Denote $h_0(u,v,x,y) = \max\{0,h(u,v,x,y)\}$, then
\begin{align}
\var(\Psi) &\leq \frac{(\pi^*)^2}{\nu^4}\sum_{x,u, y,v \in V}\pi(x)P(x,u)\pi(y)P(y,v) h_0(u,v,x,y).\label{eqn:varPsiineq1}
\end{align}

A careful analysis of the function $h_0$ leads to the following result.
\begin{lemma}\label{lemma:valueA}
In $\NVM(V,P,p)$ we have
$\Var(\Psi)  \leq 16(\pi^*)^2\nu^{-4}\sigma^2.$
\end{lemma}

\begin{proof}[Proof of Theorem~\ref{thm:normalApprox}]  Note that  Lemma~\ref{lemma:valueA} is exactly the bound of equation~\eqref{eqn:boundVarDelta1}, which applied to \eqref{eqn:hardcondition1} leads to the condition stated in equation \eqref{eqn:hardcondition2}, and thus Theorem~\ref{thm:normalApprox} is proved. 
\end{proof}

We remark that to obtain better bounds on $\var(\Psi)$ we  need to rely on equation~\eqref{eqn:varPsiineq1} in combination with estimates of $h_0$ which depend on the geometry of the graph/transition matrix. Indeed, this is how we proceed for the 2D  torus and the cycle.

\subsection{Proof of Lemma~\ref{lemma:valueA}}

The proof of Lemma~\ref{lemma:valueA} needs a subtle analysis of the function $h_0$ defined in \eqref{eq:hasCov}. Our analysis makes heavy use of the dual process in discrete time-- recall the notation for meeting and absorption times from Section~\ref{sec:properties}.  For ease of notation in this section we shall write $M^{x,u}$ for $M^{Y(x),Y(u)}$ and $\tau^x$ for $\tau^{Y(x)}$. For $u,v,x,y\in V$, recall the definition of $h_0(u,v,x,y) $ and set
\begin{align}\label{eqn:defiTuvxy}
  \mathcal{T}^{u,v,x,y}:=\{\tau^u,\tau^v,\tau^x,\tau^y,M^{u,v}, M^{v,x},M^{u,x},M^{u,y},M^{v,y},M^{x,y}  \}
\end{align}
The following proposition is crucial.
\begin{proposition}\label{P:martmethod}
Fix $u,v,x,y\in V$ and denote by $T_1<T_2<T_3$ the three smallest elements of $\mathcal{T}^{u,v,x,y}$. Then
\begin{align}\label{eqn:martmethodhbound}
h_0(u,v,x,y)&\le \mathbb{P}(T_1=M^{x,y},\,T_2=\tau^x,\,T_3=M^{u,v})+\mathbb{P}(T_1=M^{x,y},\,T_2=M^{u,v})\nonumber\\
&\phantom{\le} +\mathbb{P}(T_1=M^{x,v}, T_2=\tau^x, T_3=M^{u,y})+\mathbb{P}(T_1=M^{x,v},T_2=M^{u,y})\nonumber\\
&\phantom{\le} +\mathbb{P}(T_1=M^{u,v}, T_2=\tau^u, T_3=M^{x,y})+\mathbb{P}(T_1=M^{u,v},T_2=M^{x,y})\nonumber\\
&\phantom{\le} +\mathbb{P}(T_1=M^{u,y}, T_2=\tau^u, T_3=M^{x,v})+\mathbb{P}(T_1=M^{u,y},T_2=M^{x,v}).\nonumber\\
\end{align}
\end{proposition}
For the proof of Lemma~\ref{lemma:valueA}, the following corollary will suffice.
\begin{corollary}\label{C:martmethod}
For each $u,v,x,y\in V$, 
$h_0(u,v,x,y)\le \mathbb{P}(M^{x,y}<\tau^x\wedge\tau^y)+\mathbb{P}(M^{x,v}<\tau^x\wedge\tau^v)+\mathbb{P}(M^{u,v}<\tau^u\wedge\tau^v)+\mathbb{P}(M^{u,y}<\tau^u\wedge\tau^y)$.
\end{corollary}

\begin{proof}[Proof of Lemma~\ref{lemma:valueA}]
 For $x,y\in V$ define $E_{x,y} := \{M^{x,y}<\tau^x\wedge\tau^y\}$, then by combining \eqref{eqn:varPsiineq1} with Corollary~\ref{C:martmethod} we obtain \begin{align*}
&\var(\Psi) \\&\leq \frac{(\pi^*)^2}{\nu^4} \sum_{x,u,y,v\in V} \pi(x)P(x,u)\pi(y)P(y,v)
\left(\mathbb{P}(E_{x,y})+\mathbb{P}(E_{x,v})+\mathbb{P}(E_{u,v})+\mathbb{P}(E_{u,y})\right).
\end{align*}
We claim that the right-hand side above is equal to $4\frac{(\pi^*)^2}{\nu^4} \sum_{x,y\in V} \pi(x)\pi(y)
\mathbb{P}(E_{x,y}).$ To see this, observe that 
\begin{align*}
    \sum_{x,u,y,v\in V} \pi(x)P(x,u)\pi(y)P(y,v)
\mathbb{P}(E_{x,v})&=\sum_{x,v\in V}\pi(x)\sum_{y\in V}\pi(y)P(y,v)\mathbb{P}(E_{x,v})\\
&=\sum_{x,v\in V}\pi(x)\pi(v)\mathbb{P}(E_{x,v}),
\end{align*}
and similarly 
\begin{align*}
\sum_{x,u,y,v\in V} \pi(x)P(x,u)\pi(y)P(y,v)
\mathbb{P}(E_{u,v})&=\sum_{u,v\in V}\pi(u)\pi(v)\mathbb{P}(E_{u,v}),\\
\sum_{x,u,y,v\in V} \pi(x)P(x,u)\pi(y)P(y,v)
\mathbb{P}(E_{u,y})&=\sum_{u,y\in V}\pi(u)\pi(y)\mathbb{P}(E_{u,y}),
\end{align*}
which proves the claim. We deduce that
\begin{align*}
\var(\Psi)&=4\frac{(\pi^*)^2}{\nu^4} \sum_{x,u,y,v\in V} \pi(x)P(x,u)\pi(y)P(y,v)
\mathbb{P}(E_{x,y})\\
&=4\frac{(\pi^*)^2}{\nu^4}\sum_{x,y\in V} \pi(x)\pi(y) \mathbb{P}(E_{x,y}) = 16\frac{(\pi^*)^2}{\nu^4}\sigma^2,
\end{align*}
where the final equality holds due to Lemma \ref{lemma:sigma2}.
\end{proof}
We now focus on proving Proposition~\ref{P:martmethod}. For that, we require some preliminary results that use similar martingale arguments. We begin with a simple lemma.
\begin{lemma}\label{L:fmart}
For distinct $w,z\in V$, define $f(w,z):=\mathbb{P}(B(w)=B(z))-1/2$ and for each $t\in\mathbb{N}_0$ set $F_t=f(Y_t(w),Y_t(z))$. Then $(F_t)$ is a martingale up to time $R=\min\{\tau^w,\tau^z,M^{w,z}\}$. Moreover, $f(w,z)=\frac12\mathbb{P}(R=M^{w,z})$.
\end{lemma}
\begin{proof}
It is immediate that $F_t$ is a martingale. We then apply optional stopping at time $R$ which is valid since the martingale is bounded and $R<\infty$ almost surely (the chain  $P$ is finite and irreducible). 
\end{proof}
\begin{lemma}\label{L:gmart}For distinct $a,c,d$, define $g(a,c,d):=\mathbb{P}(B(d)\neq B(a),B(d)\neq B(c))-\mathbb{P}(B(d)\neq B(a))\mathbb{P}(B(d)\neq B(c))$ and for each $t\in\mathbb{N}_0$ set $G_t:=g(Y_t(a),Y_t(c),Y_t(d))$. Then $(G_t)$ is a martingale up to time $T:=\min\{\tau^a,\tau^c,\tau^d,M^{a,c},M^{a,d},M^{c,d}\}$. Moreover 
\[g(a,c,d)=\E[G_T\indic{T=\tau^d}+G_T\indic{T=M^{a,c}}].\]
\end{lemma}
\begin{proof}The fact that $G_t$ is a martingale is immediate. 
We apply optional stopping at time $T$ to obtain $g(a,c,d)=\E[G_T]$. Note that almost surely,
\[
G_T\indic{T=\tau^a}=\left(\frac12\mathbb{P}(B(Y_T(d))\neq B(Y_T(c)))-\frac12\mathbb{P}(B(Y_T(d))\neq B(Y_T(c)))\right)\indic{T=\tau^a}=0.
\]Similarly $G_T\indic{T=\tau^c}=0$ almost surely. Furthermore, $G_T\indic{T=M^{a,d}}=G_T\indic{T=M^{d,c}}=0$. Hence $g(a,c,d)=\E[G_T\indic{T=\tau^d}+G_T\indic{T=M^{a,c}}].$
\end{proof}
Combining these two lemmas we can obtain the following:
\begin{corollary}\label{C:gmart}
For distinct $a,c,d\in V$ set $G_t:=g(Y_t(a),Y_t(c),Y_t(d))$. Define times \begin{align*}T&:=\min\{\tau^a,\tau^c,\tau^d,M^{a,c},M^{a,d},M^{c,d}\},\\R&:=\inf\{t>T:\,t\in \{\tau^a,\tau^c,\tau^d,M^{a,c},M^{a,d},M^{c,d}\}\}.\end{align*} Then 
\[
g(a,c,d)\le (1/4)\mathbb{P}(T=\tau^d,R=M^{a,c})+\mathbb{P}(T=M^{a,c}).
\]
\end{corollary}
\begin{proof}
This follows by applying Lemmas~\ref{L:fmart} and~\ref{L:gmart}, the strong Markov property at time $T$, and the bound $G_t\le 1$ which holds for all times $t$. Indeed, we have 
\[
\indic{T=\tau^d}G_T=\frac12\indic{T=\tau^d} f(Y_T(a),Y_T(c)),
\]
and by the strong Markov property at time $T$ and Lemma~\ref{L:fmart}
\[
\mathbb{E}[\indic{T=\tau^d}f(Y_T(a),Y_T(c))]=\frac12\mathbb{P}(T=\tau^d,R=M^{a,c}).
\]
Thus by Lemma~\ref{L:gmart} and the bound $G_T\le 1$,
\[
g(a,c,d)\le \frac14\mathbb{P}(T=\tau^d,R=M^{a,c})+\mathbb{P}(T=M^{a,c}).\qedhere
\]
\end{proof}
\begin{lemma}\label{L:hmart}For distinct $u,v,x,y\in V$, define $H_t:=h(Y_t(u),Y_t(v),Y_t(x),Y_t(y))$ and recall the definition of $\mathcal{T}^{u,v,x,y}$ in equation~\eqref{eqn:defiTuvxy}. Then $(H_t)$ is a martingale up to time $\varrho:=\min \mathcal{T}^{u,v,x,y}$. Moreover \[
h(u,v,x,y)=\mathbb{E}\left[H_\varrho\indic{\varrho\in \{M^{x,y},M^{x,v},M^{u,v},M^{u,y}\}}\right].
\]
\end{lemma}
\begin{proof}
The fact that $H_t$ is a martingale is immediate. Applying optional stopping gives $h(u,v,x,y)=\mathbb{E}[H_\varrho]$. However note that
almost surely \[
H_\varrho\indic{\varrho=\tau^u}=\left(\frac12\mathbb{P}(B(Y_\varrho(y))\neq B(Y_\varrho(v)))-\frac12\mathbb{P}(B(Y_\varrho(y))\neq B(Y_\varrho(v)))\right)\indic{\varrho=\tau^u}=0,
\]
and similarly $H_\varrho\indic{\varrho=\tau^v}=H_\varrho\indic{\varrho=\tau^x}=H_\varrho\indic{\varrho=\tau^y}=0$ almost surely.
Furthermore, $H_\varrho\indic{\varrho=M^{x,u}}=H_\varrho\indic{\varrho=M^{y,v}}=0$. This gives the stated identity.
\end{proof}
\begin{proof}[Proof of Proposition~\ref{P:martmethod}] First observe that if $h<c$ for some $c>0$ then it follows that $h_0<c$. From Lemma~\ref{L:hmart} we have
\begin{align*}&h(u,v,x,y)\\&=\mathbb{E}\left[H_{T_1}\indic{T_1=M^{x,y}}+H_{T_1}\indic{T_1=M^{x,v}}+H_{T_1}\indic{T_1=M^{u,v}}+H_{T_1}\indic{T_1=M^{u,y}}\right]\\
&=\mathbb{E}\Big[g(Y_{T_1}(u),Y_{T_1}(v),Y_{T_1}(x))\indic{T_1=M^{x,y}}+g(Y_{T_1}(u),Y_{T_1}(y),Y_{T_1}(x))\indic{T_1=M^{x,v}}\\&\phantom{=\mathbb{E}\Big[}+g(Y_{T_1}(x),Y_{T_1}(y),Y_{T_1}(u))\indic{T_1=M^{u,v}}+g(Y_{T_1}(x),Y_{T_1}(v),Y_{T_1}(u))\indic{T_1=M^{u,y}}\Big].
\end{align*}
By the strong Markov property at time $T_1$ together with Corollary~\ref{C:gmart} we obtain the claimed bound.
\end{proof}

\section{Torus and Cycle - Proof of Proposition \ref{P:grid}}\label{sec:toruscycle}
\subsection{The Torus}
In this section, we consider a graph $G_n = (V_n,E_n)$ the torus  $\mathbb{T}_n$ on $n$ vertices, and we let $P_n$ be the simple random walk on $\mathbb{T}_n$.
\begin{proposition}\label{thm:resultTorus}
There exists a universal constant $C>0$ (not depending on $n$) such that for any torus $\mathbb{T}_n$ with $n\geq 9$,
\begin{align*}
\Var(\Psi) \leq C \left(\frac{1}{\sqrt{n}}+\frac{\sigma^2}{(\log n)^2}\right).
\end{align*}
\end{proposition}

For the proof of this proposition we will consider the continuous-time version of the dual process (see Remark~\ref{remark:continuoustimedual}), where particles move at rate $1$ in the graph, and are absorbed at rate $\delta = p/(1-p)$. Further, recall that walks move independently until they meet (or are absorbed), and so couplings with the process $\RW(m)$ of $m$ independent random walks as defined at the beginning of Section~\ref{sec:properties} will feature in our analysis, indeed the following two results about $\RW(2)$ on the torus are needed.

\begin{lemma}\label{lemma:tailmeetTorus1}
 Let $(X_t,Y_t)$ be a realisation of $\RW(2)$ on the torus in continuous time. Let $M^{X,Y}$ denote the meeting time of the two walks. Then for all $t\geq 0$,
\begin{align*}
\Prob^{\RW(2)}_{\pi^2}(M^{X,Y}\leq t)\leq (2t+1)\nu^2.
\end{align*}

\end{lemma}
\begin{lemma}\label{lemma:tailmeetTorus2}
Let $(X_t,Y_t)$ be a realisation of $\RW(2)$ on the torus in continuous time. Let $x$ and $y$ be two adjacent vertices on the torus. Then there exists a universal constant $C>0$ (not depending on $n$) such that for each $t\geq e$, 
\begin{align*}
\Prob^{\RW(2)}_{(x,y)}(M^{X,Y}>t)\leq C/\log (t).
\end{align*}
\end{lemma}

\begin{proof}[Proof of Proposition~\ref{P:grid} Item \ref{thm:grid2}]
To start with, note that the case $p_nn\log n \to 0$ follows immediately from Proposition~\ref{prop:small_p}, since the meeting time of two independent stationary walks is $\Theta(n \log n)$. 

Now, suppose that $p_nn\log n \to \infty$. Recall that $p_n$ is non-increasing hence $\lim_n p_n$ exists. If $\lim_n p_n>0$ then the result follows from Corollary~\ref{cor:normalApprox}. So suppose $\lim_n p_n=0$. By Theorem~\ref{thm:normalApprox} it suffices to show
\[
\frac np\left(\frac{\pi^*}{\sigma} \right)^3 + \frac{\nu^2}{p\sigma^2}\sqrt{\var({\Psi})}\to0 \quad \text{as }n\to\infty,
\]
(to ease notation, we are not writing the subindex $n$) i.e.
\begin{align}\notag
\frac1{n^2p\sigma^3}+\frac{\sqrt{\var(\Psi)}}{np\sigma^2}\to0 \quad \text{as }n\to\infty.
\end{align}
Eventually $p\le 1/2$ so by Lemma~\ref{lemma:easyVar1}, $\sigma^2\gtrsim\frac1{pt_\mathrm{hit}}\gtrsim \frac1{pn\log n}$ (see \cite[Proposition 10.13]{levin2009markov}), hence 
\[
\frac1{n^2p\sigma^3}\lesssim\sqrt{p(\log n)^3/n}\to0\quad\text{as }n\to\infty.
\]
By Proposition~\ref{thm:resultTorus} it remains to show that 
\begin{align*}
\frac{\sqrt{\frac1{\sqrt n}+\frac{\sigma^2}{(\log n)^2}}}{np\sigma^2}\le \frac1{n^{5/4}p\sigma^2}+\frac1{np\sigma\log n}\to0\quad\text{as }n\to\infty.
\end{align*}
Using again $\sigma^2\gtrsim \frac1{pn\log n}$ we have the bound
\[
\frac1{n^{5/4}p\sigma^2}+\frac1{np\sigma\log n}\lesssim \frac{\log n}{n^{1/4}}+\frac1{\sqrt{np\log n}}\to0\quad\text{as }n\to\infty,
\]
by the assumption $pn\log n\gg1$, completing the proof.
\end{proof}

We proceed to prove Proposition~\ref{thm:resultTorus}. Our proof uses the dual process.  For ease of notation in we shall write $M^{x,u}$ for $M^{Y(x),Y(u)}$ and $\tau^x$ for $\tau^{Y(x)}$, etc.

\begin{proof}[Proof of Proposition~\ref{thm:resultTorus}]
To bound $h_0$ we use Proposition~\ref{P:martmethod}. Consider two terms at the top of the right-hand side of equation~\eqref{eqn:martmethodhbound}, given by $$\mathbb{P}(T_1=M^{x,y},\,T_2=\tau^x,\,T_3=M^{u,v})+\mathbb{P}(T_1=M^{x,y},\,T_2=M^{u,v}).$$

Here we are considering the continuous-time version of the dual. Note the previous events are disjoint and establish that the meeting of $x$ and $y$ occurs before any other meeting or absorption, so, in particular, we have
\begin{align*}
\mathbb{P}(T_1=M^{x,y},\,T_2=\tau^x,\,T_3=M^{u,v})+\mathbb{P}(T_1=M^{x,y},\,T_2=M^{u,v}) \leq \Prob(T_1 = M^{x,y}).
\end{align*}
Note that up to time $T_1$ no meeting has occurred, so the four particles involved move like independent random walks up to this time,  thus it is natural to couple $(Y_t(x),Y_t(y),Y_t(u),Y_t(v))$ up to time $T_1$ with an independent random walks process $\RW(4)$ that is stopped when a meeting occurs or when an independent exponential clock of rate $4p/(1-p)$ rings (representing the first absorption time of the four walks). Let $(X_t,Y_t,U_t,V_t)$ be a $\RW(4)$ process, let $T$ be the time of the first meeting between any pair of the walks in $\RW(4)$, let $M^{X,Y}$ be the meeting time between $(X_t)$ and $(Y_t)$, and let $Z$ be an exponential random variable of parameter $4\delta$ with $\delta = p/(1-p)$, which is independent of the walks. Then we have

\begin{align*}
    \Prob(T_1 = M^{x,y}) &= \Prob_{(x,y,u,v)}^{\RW(4)}(M^{X,Y}=T, M^{X,Y}<Z) \\&\leq \Prob_{(x,y,u,v)}^{\RW(4)}(M^{X,Y}< M^{X,U} \wedge M^{Y,V} \wedge Z )
\end{align*}

Repeating the same argument to the other terms of equation~\eqref{eqn:martmethodhbound}, using reversibility of $P$, and plugging into equation~\eqref{eqn:varPsiineq1} yields
\begin{align}
&\Var(\Psi) \notag\\&\leq 4\frac{(\pi^*)^2}{\nu^4} \sum_{x,y,u,v\in V} \pi(x)P(x,u)\pi(y)P(y,v)\Prob_{(x,y,u,v)}^{\RW(4)}(M^{X,Y}\leq M^{X,U} \wedge M^{Y,V} \wedge Z )\nonumber\\
&=4 \sum_{x,y,u,v\in V} \frac{\ind{\{x\sim u\}} \ind{\{y\sim v\}}}{16n^2}\Prob_{(x,y,u,v)}^{\RW(4)}(M^{X,Y}\leq M^{X,U} \wedge M^{Y,V} \wedge Z ).\label{eqn:randomvjr82844} 
\end{align}
We proceed to bound the probability in \eqref{eqn:randomvjr82844}. We write members of $V^4$ in bold, i.e. $\boldsymbol{x}= (x,y,u,v)\in V^4$, and write $\boldsymbol{X}_t$ for $(X_t,Y_t,U_t,V_t)$. 
As previously, $M^{X,Y}$ refers to the meeting time of random walks $X$ and $Y$.

Let $\boldsymbol{x}\in V^4$. If $x \neq y$ then $\Prob^{\RW(4)}_{\bf x}(M^{X,Y} =0, M^{X,Y} \leq M^{X,U} \wedge M^{Y,V}) = 0$, so in such case we denote by $f_{\boldsymbol{x}}(t)$ the derivative of \[t\mapsto   \Prob^{\RW(4)}_{\bf x}(M^{X,Y} \leq t, M^{X,Y} \leq M^{X,U} \wedge M^{y,v}),\]
which exists since the meeting between two particles is a continuous random variable when the particle start in different locations. 

Then we can write $f_{\boldsymbol{x}}$ in terms of the four independent random walks. Indeed:
\begin{align*}
f_{\boldsymbol{x}}(t) = \frac{1}{2}\Prob^\mathrm{RW(4)}_{\boldsymbol{x}}(\{M^{X,Y}> t,   \,M^{X,U}> t,\, M^{Y,V}> t\} \cap \{X_t \sim Y_t\}),
\end{align*}
which holds because the instantaneous probability that $X_t$ and $Y_t$ meet for first time at time $t$ is exactly the probability that they do not meet up to time $t$, then at time $t$ one walker moves to the location of the other (so they have to be neighbours just prior to time $t$). The rate of jumping to an adjacent neighbour is $1/4$, giving the equation above.

Write $\mathcal E_t = \{M^{X,Y}> t,\,   M^{X,U}> t,\, M^{Y,V}> t\}$; then $f_{\boldsymbol{x}}(t)$ can be written as
\begin{align*}
f_{\boldsymbol{x}}(t) = \frac12\Prob^\mathrm{RW(4)}_{\boldsymbol{x}}(\mathcal E_t, X_t \sim Y_t).
\end{align*}
By using the Markov property and reversibility we have
\begin{align}
f_{\boldsymbol{x}}(t) &= \frac12\sum_{\boldsymbol{x'}\in V^4}\Prob^\mathrm{RW(4)}_{\boldsymbol{x}}(\mathcal E_{t/2}, \boldsymbol{X}_{t/2} = \boldsymbol{x'})\Prob^\mathrm{RW(4)}_{\boldsymbol{x'}}(\mathcal E_{t/2}, X_{t/2} \sim Y_{t/2})\nonumber\\
&= \frac12\sum_{\boldsymbol{x'}\in V^4}\Prob^\mathrm{RW(4)}_{\boldsymbol{x'}}(\mathcal E_{t/2}, \boldsymbol{X}_{t/2} = \boldsymbol{x})\Prob^\mathrm{RW(4)}_{\boldsymbol{x'}}(\mathcal E_{t/2}, X_{t/2} \sim Y_{t/2})\nonumber\\
&\leq \frac12\sum_{\boldsymbol{x'}\in V^4}\Prob^\mathrm{RW(4)}_{\boldsymbol{x'}}(M^{X,U}>t/2, M^{Y,V}>t/2, \boldsymbol{X}_{t/2} = \boldsymbol{x})\notag\\&\phantom{\leq \frac12\sum_{\boldsymbol{x'}\in V^4}}\times\Prob^\mathrm{RW(4)}_{\boldsymbol{x'}}(M^{X,Y}>t/2, X_{t/2} \sim Y_{t/2})\nonumber\\
\begin{split}&=\frac12\sum_{\boldsymbol{x'}\in V^4}\Prob^\mathrm{RW(2)}_{(x',u')}(M^{X,U}>t/2, (X_{t/2},U_{t/2}) = (x,y))\\
&%
\phantom{ =\frac12\sum_{\boldsymbol{x'}\in V^4}}
\times \Prob^\mathrm{RW(2)}_{(y',v')}( M^{Y,V}>t/2, (Y_{t/2}, V_{t/2})= (y,v))\\
&\phantom{ =\frac12\sum_{\boldsymbol{x'}\in V^4}} 
\times \Prob^\mathrm{RW(2)}_{(x',y') }(M^{X,Y}>t/2, X_{t/2} \sim Y_{t/2}),\end{split}\label{eqn:randomv8v84811a}
\end{align}
where $\Prob^\mathrm{RW(2)}$ is the probability measure associated with two independent random walks.

Denote by $\mathcal F_t$ the event $\{M^{X,Y}>t, X_t \sim Y_t\}$, then from \eqref{eqn:randomv8v84811a} we have
\begin{align*}
\sum_{\boldsymbol{x}\in V^4}\frac{\ind{\{x\sim u\}} \ind{\{y\sim v\}}}{16n^2}f_{\boldsymbol{x}}(t)&\leq   \frac{2}{64n^2}  \sum_{\boldsymbol{x'}\in V^4}\Prob^\mathrm{RW(2)}_{(x',u') }(\mathcal F_{t/2})\Prob^\mathrm{RW(2)}_{(y',v') }(\mathcal F_{t/2})\Prob^\mathrm{RW(2)}_{(x',y') }(\mathcal F_{t/2})\\
&= \frac{1}{32}\sum_{(x',y')\in V^2} \Prob^\mathrm{RW(2)}_{(x',\pi) }(\mathcal F_{t/2})\Prob^\mathrm{RW(2)}_{(y',\pi) }(\mathcal F_{t/2})\Prob^\mathrm{RW(2)}_{(x',y') }(\mathcal F_{t/2}).
\end{align*}
Since the torus is transitive, we have that $\Prob^\mathrm{RW(2)}_{(x',\pi)}(\mathcal F_{t/2})$ does not depend on $x'$, so $\Prob^\mathrm{RW(2)}_{(x',\pi)}(\mathcal F_{t/2}) = \Prob^\mathrm{RW(2)}_{(\pi,\pi)}(\mathcal F_{t/2})$, concluding from the previous equation that 
\begin{align}
\sum_{\boldsymbol{x}\in V^4}\frac{\ind{\{x\sim u\}} \ind{\{y\sim v\}}}{16n^2}f_{\boldsymbol{x}}(t) 
&\leq  \frac{1}{32}\sum_{(x',y')\in V^2} \Prob^\mathrm{RW(2)}_{(\pi,\pi)}(\mathcal F_{t/2})\Prob^\mathrm{RW(2)}_{(\pi,\pi)}(\mathcal F_{t/2})\Prob^\mathrm{RW(2)}_{(x',y')}(\mathcal F_{t/2})\nonumber\\
&= \frac{n^2}{32}\Prob^\mathrm{RW(2)}_{(\pi,\pi)}(\mathcal F_{t/2})^3.\label{eqn:randomudyg131}
\end{align}
Using reversibility again we obtain
\begin{align*}
\Prob^\mathrm{RW(2)}_{(\pi,\pi)}(\mathcal{F}_{t/2}) &= \sum_{x',y' \in V} \frac{1}{n^2}\Prob^\mathrm{RW(2)}_{(x',y')}(\mathcal F_{t/2}) \\
&= \sum_{x',y' \in V}\sum_{x,y \in V}\ind{\{x\sim y\}} \frac{1}{n^2}\Prob^\mathrm{RW(2)}_{(x',y')}(M^{X,Y}>t/2, X_{t/2} = x, Y_{t/2} = y)\\
& =\sum_{x',y' \in V}\sum_{x,y \in V}\ind{\{x\sim y\}} \frac{1}{n^2}\Prob^\mathrm{RW(2)}_{(x,y)}(M^{X,Y}>t/2, X_{t/2} = x', Y_{t/2} = y') \\
&= \sum_{x,y \in V}\ind{\{x\sim y\}} \frac{1}{n^2}\Prob^\mathrm{RW(2)}_{(x,y)}(M^{X,Y}>t/2)  \\&= \frac{4}{n}\sum_{x,y}\mu(x,y)\Prob^\mathrm{RW(2)}_{(x,y)}(M^{X,Y}>t/2).
\end{align*}
By denoting $\Prob^\mathrm{RW(2)}_{\mu}(\cdot) = \sum_{x,y}\mu(x,y)\Prob^\mathrm{RW(2)}_{(x,y)}(\cdot)$, and substituting the previous equality into \eqref{eqn:randomudyg131} we have
\begin{align}
 \sum_{\boldsymbol{x}\in V^4}\frac{\ind{\{x\sim u\}} \ind{\{y\sim v\}}}{16n^2}f_{\boldsymbol{x}}(t)
&\leq \frac{1}{8} \Prob^\mathrm{RW(2)}_{\mu}(M^{X,Y}>t/2)^2 \Prob^\mathrm{RW}_{(\pi,\pi)}(\mathcal F_{t/2}).\label{eqn:importantdensitytorus}
\end{align}

Going back to equation~\eqref{eqn:randomvjr82844}, by \eqref{eqn:importantdensitytorus} we have
\begin{align*}
\Var(\Psi)&\le 4 \sum_{x,y,u,v\in V} \frac{\ind{\{x\sim u\}} \ind{\{y\sim v\}}}{16n^2}\Prob_{(x,y,u,v)}^{\RW(4)}(M^{X,Y}\leq M^{X,U} \wedge M^{Y,V} \wedge Z )\\
&\le4 \sum_{x,y,u,v\in V} \frac{\ind{\{x\sim u\}} \ind{\{y\sim v\}}}{16n^2} \Prob_{(x,y)}^{\RW(2)}(M^{X,X}\leq 2\sqrt n)\\&\phantom{\le}+4 \sum_{x,y,u,v\in V} \frac{\ind{\{x\sim u\}} \ind{\{y\sim v\}}}{16n^2} \Prob_{(x,y,u,v)}^{\RW(4)}(2\sqrt n\le M^{X,Y}\leq M^{X,U} \wedge M^{Y,V} \wedge Z )\\
&=\frac{1}{4}\Prob^{\RW(2)}_{\pi^2}(M^{X,Y}\le 2\sqrt n) +4 \sum_{x,y,u,v\in V} \frac{\ind{\{x\sim u\}} \ind{\{y\sim v\}}}{16n^2}\int_{2\sqrt n}^\infty f_{\boldsymbol{x}}(t)e^{-4\delta t}\,dt\\
&\leq \cdot\frac{5}{4\sqrt n} + 4\int_{2\sqrt n}^{\infty} \frac{1}{8} \Prob^{\RW(2)}_{\mu}(M^{X,Y}>t/2)^2\,  \Prob^{\RW(2)}_{(\pi,\pi)}(\mathcal F_{t/2})e^{-4\delta t}\,dt\\
&\leq \frac{5}{4\sqrt n} + c\int_{2\sqrt n}^{\infty} (\log n)^{-2}\,  \Prob^{\RW(2)}_{(\pi,\pi)}(\mathcal F_{t/2})e^{-4\delta t}\,dt,
\end{align*}
for a universal constant $c>0$. The bound on $\Prob^\mathrm{RW(2)}_{\pi^2}(M^{X,Y}\le 2\sqrt n)$ follows from Lemma~\ref{lemma:tailmeetTorus1} and the bound on $\Prob^\mathrm{RW(2)}_{\mu}(M^{X,Y}>t/2)^2$ is from Lemma~\ref{lemma:tailmeetTorus2}. 
Now recall that $\frac{1}{8}\Prob^\mathrm{RW}_{(\pi,\pi)}(\mathcal F_{t/2})$ is the density of $M^{X,Y}$ at time $t/2$ when both particles start independently with distribution $\pi$, which we shall denote by $g(t/2)$. Then,
\begin{align*}
&\var(\Psi)\\&\le \frac{5}{4\sqrt n} + 8c\int_{2\sqrt n}^{\infty} (\log n)^{-2}\,  g(t/2)e^{-4\delta t}\,dt\le \frac{5}{4\sqrt n} + 8c\int_{0}^{\infty} (\log n)^{-2}\,  g(s)e^{-2\delta s}\,ds\\
&\leq\frac{5}{4\sqrt n} + 8c(\log n)^{-2}\,\mathbb{E}_{(\pi,\pi)}^\mathrm{RW(2)}\big[e^{-2\delta M^{X,Y}}\big]= \frac{5}{4\sqrt n} + 8c(\log n)^{-2}\,\Prob_{\pi^2}(M^{XY}< Z')\\
&\leq \hat c\left(\frac1{\sqrt n}+\sigma^2(\log n)^{-2}\right),
\end{align*}
where $Z'$ is an exponential random variable of parameter $2\delta$ (independent of everything), $c>0$ and $\hat c>0$ are universal constants, and the last inequality holds due to Lemma~\ref{lemma:sigma2}.
\end{proof}
We now complete this section by presenting the proofs of Lemmas~\ref{lemma:tailmeetTorus1} and~\ref{lemma:tailmeetTorus2}.
\begin{proof}[Proof of Lemma~\ref{lemma:tailmeetTorus1}] To ease notation, and since there are no other processes in the proof, we omit the superscript $\RW(2)$ in the probability measures. Now, denote by $N(t)$ the total number of jumps made by the two walks by time $t$. Then we claim that for each $t\ge0$ and $j\in\mathbb{N},$
\begin{align*}
\Prob_{\pi^2}(M^{X,Y}\leq t\mid N(t) = j) \leq (j+1)\nu^2.
\end{align*}
To see this, denote by $T_1,\ldots, T_j$ the jump times. Then at these times the distribution of the locations of the walks is exactly the stationary distribution by the strong Markov property (and so the probability they are equal is $\nu^2=\sum_{x\in V} \pi(x)^2$). Thus
\begin{align*}
\Prob_{\pi^2}(M^{X,Y}\leq t\mid N(t) = j) \leq \Prob_{\pi^2}(X_0=Y_0)+\sum_{i=1}^j \Prob_{\pi^2}(X_{T_i}=Y_{T_i}) = (j+1)\nu^2.
\end{align*}
We conclude that $\Prob_{\pi^2}(M^{X,Y}\leq t)\leq \nu^2(\E_{\pi^2}(N(t)+1)) = \nu^2(2t+1)$ since $N(t)$ has Poisson distribution of mean $2t$.
\end{proof}
\begin{proof}[Proof of Lemma~\ref{lemma:tailmeetTorus2}]
 As in the proof of Lemma~\ref{lemma:tailmeetTorus1} we omit the superscript $\RW(2)$ in the probability measure.
The first step is to assume that the random walks move on $\mathbb{Z}^2$ instead of $\mathbb{T}_n$, which clearly provides an upper-bound on the quantity of interest. Recall that $x=X_0$ and $y=Y_0$ are adjacent vertices in $\Z^2$, and by transitivity we can assume that $y=(0,0) \in \Z^2$ and $x=(1,0)\in \Z^2$. Also, by further using the symmetry of $\Z^2$, we can assume that only one walk moves but at rate 2, while the other is fixed. Then, denote by $Z_t$ a simple random walk on $\mathbb{Z}^2$ moving at rate $2$, let $o$ be the origin $(0,0) \in \mathbb{Z}^2$ and $x = (1,0)\in \Z^2$, and let $H_o$ be the hitting time of $o$ by the random walk $Z_t$. By our previous discussion, we deduce that
\begin{align*}
\Prob_{(x,y)}(M^{X,Y}>t) \leq \Prob_{x}(H_o>t),
\end{align*}
 We condition on the number of jumps $N(t)$ the walk $Z_t$ performs up to time $t$ which is distributed as a Poisson random variable with mean $2t$. Denote by $\widehat H_o$ the number of jumps until $Z_t$ hit vertex $o$, and by $M(j)$ the number of times the random walk jumps into vertex $o$ up to (and including) the $j$-th jump. Then
\begin{align}
\Prob_x(H_o>t\mid N(t)=j) = \Prob_x(\widehat H_o > j) = 1-\Prob_x(M(j)>0).
\end{align}
Note that conditioning on $\{N(t) = 2j\}$ or $\{N(t) = 2j+1\}$ is the same since $\mathbb{Z}$ is bipartite, i.e.\! $\Prob_x(M(2j)>0)=\Prob_x(M(2j+1)>0)$, so without loss of generality we focus just on $\Prob_x(M(2j)>0)$. 

We claim that for $j\geq 1$,
\begin{align}\label{eqn:randomjgurhu1931a}
1-\Prob_x(M(2j)>0) \leq \frac{1}{2(1+c'\log(j))},
\end{align}
for some constant $c'>0$.
By using this we have
\begin{align*}
\Prob_{x}(H_o>t) &= \sum_{j=0}^{\infty} (1-\Prob_{x}(M(2j)>0))\Prob_x(N(t)\in \{2j,2j+1\})\\
&\leq \frac12\sum_{j=0}^{\infty} \frac{\Prob_x(N(t)\in \{2j,2j+1\})}{1+c'\log(j+1)} \\
&\le \frac12\Prob_x(N(t)\in\{0,1\})+\frac12\sum_{j=2}^\infty\frac{\Prob_x(N(t)=j)}{1+c'\log((j+1)/2)}\\
&\le \frac12\Prob_x(N(t)\in\{0,1\})+\frac12\sum_{j=2}^{\sqrt t}\Prob_x(N(t)=j)+\frac12\sum_{j=\sqrt t+1}^\infty \frac{\Prob_x(N(t)=j)}{1+c'\log(\sqrt t/2 + 1)}\\
&\le \frac12\Prob_x(N(t)\le \sqrt t)+\frac1{2(1+c'\log(\sqrt t/2 + 1))}
\leq \frac{C}{\log t}
\end{align*}
for some $C>0$, using in the last bound Chebyshev's inequality, for example.
 
 To verify equation~\eqref{eqn:randomjgurhu1931a} we start by noting that
\begin{align*}
\Prob_x(M(2j)>0) = \frac{\E_x(M(2j))}{\E_x(M(2j)|M(2j)>0)}.
\end{align*}
For the denominator, we can assume that $Z_t$ hit vertex $o$ for first time at time-step  $1$, obtaining the following upper bound $\E_x(M(2j)\mid M(2j)>0)\leq \E_o(M(2j-1)) = \sum_{i=0}^{2j-1}P^i(o,o)$.

For the numerator, note that $\E_x(M(2j)) = \E_z(M(2j))$ for any $z \sim o$, by the symmetry of $\mathbb{Z}^2$ (recall that $x\sim o$), and so
\begin{align*}
\E_x(M(2j)) = \frac{1}{d}\sum_{z:z\sim o} \E_z(M(2j)) =\sum_{z:z\sim o} P(o,z)\E_z(M(2j)) = \sum_{i=1}^{2j+1} P^i(o,o).
\end{align*}
Then
\begin{align*}
1-\Prob_{x}(M(j)>0) &\leq 1- \frac{\sum_{i=1}^{2j+1} P^i(o,o)}{\sum_{i=0}^{2j-1}P^i(o,o)}= \frac{1-P^{2j}(o,o)-P^{2j+1}(o,o)}{\sum_{i=0}^{2j-1}P^i(o,o)}.
\end{align*}
Note that $P^{2j}(o,o)+P^{2j+1}(o,o) \leq 1/2$, so 
\begin{align*}
1-\Prob_{x}(M(2j)>0)\leq \frac{1}{2\sum_{i=0}^{2j-2}P^i(o,o)} = \frac{1}{2\sum_{i=0}^{j-1}P^{2i}(o,o)}\leq \frac{1}{2(1+ c'\log(j))},
\end{align*}
where $c'>0$ is a constant independent of $i$. The last equality follows since  $P^{2i}(o,o)\geq \frac{c'}{i}$ for all $i\geq 1$ where $c'>0$ is independent of the number of vertices and $i$. To see this, note that for a random walk on $\mathbb{Z}^2$, $P^{2i}(o,o)\geq \frac{1}{\pi i}-\frac{c}{i^2}$ by Theorem~\ref{thm:LCLT}, where $c>0$ is independent of $i$; so by choosing $c'$ small enough we have $P^{2i}(o,o)\geq \frac{c'}{i}$, which gives
\[
\sum_{i=0}^{j-1}P^{2i}(o,o) \geq 1+ \sum_{i=1}^{j-1}P^{2i}(o,o) \geq  1+ \sum_{i=1}^{j-1} c'/i \geq 1 + c'\log j.\qedhere\]
\end{proof}

\subsection{The Cycle $C_n$}Our interest in this section is on the noisy voter model on the cycle\footnote{$C_n$ has vertex set $[n]$ and edge set $\{\{1,2\},\{2,3\},\ldots,\{n-1,n\},\{n,1\}\}$} $C_n$ with $P_n$ corresponding to the simple random walk. Specifically, we shall prove Item~\ref{thm:grid1} of Proposition~\ref{P:grid}. 
We first obtain the precise value of the variance, $\sigma^2=\var(S_n)$. 

\begin{lemma}[Cycle variance]\label{L:cyclevar} 
For each $p\in(0,1)$ set $\theta=\frac{1-\sqrt{p(2-p)}}{1-p}\in(0,1)$. Then
\[
\sigma^2=\frac1{4n}\left(1+\frac{2\theta(1-\theta^{n-1})}{(1-\theta)(1+\theta^n)}\right).
\]
\end{lemma}
\begin{proof}
We appeal to martingale arguments. By Lemma~\ref{lemma:sigma2}, $\sigma^2=\frac{1}{4n^2}\sum_{x,y\in V}\Prob(E_{xy})$ and recall that $E_{x,y}=\Prob(M^{Y(x),Y(y)}< \tau^{Y(x),Y(y)})$. Here we are considering the discrete-time dual process. For $t\in\mathbb{N}_0$, set $X_t$ to be the clockwise distance between the two particles $Y(x)$ and $Y(y)$ initialised at $x$ and $y$ in the dual process just after the $t^\mathrm{th}$ time that one of these two particles is selected to move/be absorbed (we are essentially looking at the dual when particles $Y(x)$ and $Y(y)$ move, ignoring the other particles in the process). Note that $(X_t)$ behaves like a discrete-time simple random walk on the integers. Let $T_{0,n}$ be the hitting time of $\{0,n\}$ by $(X_t)$ and $G$ an independent Geometric random variable with parameter $p$. Then we deduce that $\sigma^2=\frac14\sum_{a,b\in\{0,\ldots,n-1\}}n^{-2}\mathbb{P}^\mathrm{RW}_{|a-b|}(T_{0,n}<G)$, where $\mathbb{P
}^\mathrm{RW}$ is the probability measure associated with $(X_t)$ (a simple random walk on $\mathbb{Z}$). We now consider the bounded process $M_t:=\theta^{X_t}(1-p)^t$. Since $\theta+\theta^{-1}=2/(1-p)$, this is a martingale up to $T_{0,n}$. Thus by the optional stopping theorem, for each $k\in\{1,\ldots,n-1\}$,
\[
\theta^{k}=\mathbb{E}^\mathrm{RW}_{k}\left[(1-p)^{T_{0,n}}\left(\indic{X_{T_{0,n}}=0}+\theta^n\indic{X_{T_{0,n}}=n}\right)\right].
\]
By symmetry 
\begin{align*}
&\mathbb{E}^\mathrm{RW}_{k}\left[(1-p)^{T_{0,n}}\left(\indic{X_{T_{0,n}}=0}+\theta^n\indic{X_{T_{0,n}}=n}\right)\right]\\&=\mathbb{E}^\mathrm{RW}_{n-k}\left[(1-p)^{T_{0,n}}\left(\indic{X_{T_{0,n}}=n}+\theta^n\indic{X_{T_{0,n}}=0}\right)\right], 
\end{align*}
thus 
\begin{align*}
&\theta^k+\theta^{n-k}\\&=\mathbb{E}^\mathrm{RW}_{k}\left[(1-p)^{T_{0,n}}\left(\indic{X_{T_{0,n}}=0}+\theta^n\indic{X_{T_{0,n}}=n}\right)\right]\\&\phantom{=}+\mathbb{E}^\mathrm{RW}_{n-k}\left[(1-p)^{T_{0,n}}\left(\indic{X_{T_{0,n}}=0}+\theta^n\indic{X_{T_{0,n}}=n}\right)\right]\\
&=\mathbb{E}^\mathrm{RW}_{k}\left[(1-p)^{T_{0,n}}\left(1+\theta^n\right)\right],
\end{align*}
i.e.\! \begin{align}\label{eq:ERWk}\mathbb{E}^\mathrm{RW}_{k}\left[(1-p)^{T_{0,n}}\right]=(\theta^k+\theta^{n-k})(1+\theta^n)^{-1}.\end{align}
To relate this to $\sigma^2$, note that we can write
\begin{align*}
\sigma^2=\frac14\sum_{a,b\in\{0,\ldots,n-1\}}n^{-2}\mathbb{E}^\mathrm{RW}_{|a-b|}\left[(1-p)^{T_{0,n}}\right]&=\frac14 n^{-2}\left(n\cdot 1+n\sum_{k=1}^{n-1}\mathbb{E}^\mathrm{RW}_{k}\left[(1-p)^{T_{0,n}}\right]\right)
\end{align*}
and so plugging in \eqref{eq:ERWk} and simplifying we obtain
\[
\sigma^2=\frac1{4n}\left(1+\frac{2\theta(1-\theta^{n-1})}{(1-\theta)(1+\theta^n)}\right).\qedhere
\]
\end{proof}

We introduce an auxiliary process which, up to a time change, is almost identical to the dual process on the cycle $C_n$.
Given $n\in\mathbb{N}$ and $p\in(0,1)$, we define CAB ({\bf c}oalescing walk with {\bf ab}sorption) as a process $({\bf X}_t)_{t\in\mathbb{N}_0}=(X_t,U_t,V_t,Y_t)_{t\in\mathbb{N}_0}$ taking values in $(C_n\cup \Delta)^4$ where $\Delta$ a cemetery state. The process starts at time $0$ from a state in $(C_n)_4$ and we say that all 4 particles are alive. At each time $t\in\mathbb{N}$, we choose uniformly an alive particle and this particle with probability $p$ gets absorbed into the cemetery state and dies, otherwise the particle jumps clockwise or anti-clockwise on the cycle with equal probability. If this jump results in particles meeting, these particles coalesce into a single alive particle (i.e. the number of alive particles decreases by 1). Any coalesced particles evolve in the same way at the same times, including having the same absorption times into the cemetery state.

The CAB process is similar to the dual process except time is sped-up, we only observe the movement of four particles, and the two stubborn vertices considered in the dual process are glued together into a single cemetery state $\Delta$. To see this, note that if we consider the dual process, then when a particle $Y(z)$ coalesces with one of the other three particles, say $Y(x)$, then $Y(z)$ follows the trajectory of $Y(x)$ (this does not change the distribution of the dual process but allows us to ignore all particles different from $Y(x),Y(y),Y(u)$ and $Y(v)$). As a result probabilities of statements involving orders of meetings/absorption times are identical in both processes.

For a CAB process $({\bf X}_t)_{t\in\mathbb{N}_0}=(X_t,U_t,V_t,Y_t)_{t\in\mathbb{N}_0}$, we use $\tau^A$ to denote the absorption time into $\Delta$ of particle $A\in\{X,U,V,Y\}$, $M^{A,B}$ to denote meeting time of particles $A,B\in\{X,U,V,Y\}$ and $\tau:=\min\{\tau^X,\tau^U,\tau^V,\tau^Y,M^{X,U},M^{V,Y}\}$. We also use $\Prob^{CAB}_{\bf x}$ to denote probability with respect to the CAB process when the four particle start from ${\bf x} \in C_n^4$. Additionally, for $a,b\in C_n$ we define $d(a,b)$ to be the clockwise distance from $a$ to $b$ and say that $(x_1,x_2,x_3,x_4)\in(C_n)_4$ is \emph{clockwise-oriented} if and only if for $i\in\{1,2\}$ and $j\notin\{i,i+1\}$, $d(x_i,x_{i+1})<d(x_i,x_j)$.

We present a refinement of Proposition~\ref{P:martmethod} to be used for the cycle, presented in terms of a CAB process. This provides us with the probabilities we need to control to bound the covariance term. Recall the function $h$ defined in \eqref{eq:hasCov}.
\begin{corollary}[Cycle covariance]\label{C:cyclecov}
Let  ${\bf x}=(x,u,v,y)\in (C_n)_4$ be clockwise-oriented. Let $({\bf X}_t)_{t\ge0}=(X_t,U_t,V_t,Y_t)_{t\ge0}$ be a CAB process initialised from $(x,u,v,y)$.  Then
\begin{align*}
&h_0(u,v,x,y)\\&\le \Prob_{\bf x}^{\mathrm{CAB}}\big(M^{X,Y}<\tau\big)
+\Prob_{\bf x}^{\mathrm{CAB}}\big(M^{U,V}<\tau^U=\min\{\tau,M^{X,Y}\}\big).
\end{align*}
\end{corollary}
\begin{proof}
Proposition~\ref{P:martmethod} gives a bound on $h_0(u,v,x,y)$ in terms of the dual process, the probabilities in this bound being identical to probabilities for the CAB process, as already remarked. Further, due to the structure of the cycle, and the clockwise orientation  of ${\bf x}$, certain probabilities are zero (those involving $T_1=M^{x,v}$ or $T_1=M^{u,y}$). The remaining probabilities can be easily bounded to give the claimed result.
\end{proof}
We control the probabilities in Corollary~\ref{C:cyclecov} using the following two propositions.
\begin{proposition}\label{P:cyclemain}Let  ${\bf x}=(x,u,v,y)\in (C_n)_4$ be clockwise-oriented and distinct with $x\sim u$, $v\sim y$ and $d:=d(y,x)\ge d(u,v)$. Let $({\bf X}_t)_{t\ge0}=(X_t,U_t,V_t,Y_t)_{t\ge0}$ be a CAB process initialised from $(x,u,v,y)$.  Then there exists a universal constant $C>0$ such that for any $p\in(0,1)$ and $n\ge 8$, 
\[\Prob^\mathrm{CAB}_{\bf x}\big(M^{X,Y}< \tau\big) \leq Cd^{-2}\left(\exp(-Cd\sqrt p)+d^{-1}\right).\]
\end{proposition}
We give the proof towards the end of this section. 

\begin{proposition}\label{P:cycle2ndprob}
Let  ${\bf x}=(x,u,v,y)\in (C_n)_4$ be clockwise-oriented and distinct with $x\sim u$, $v\sim y$ and $d(y,x)\ge d(u,v)$. Let $({\bf X}_t)_{t\ge0}=(X_t,U_t,V_t,Y_t)_{t\ge0}$ be a CAB process initialised from $(x,u,v,y)$.  Then for any $p\in(0,1)$,
\[\Prob_{\bf x}^{\mathrm{CAB}}\big(M^{U,V}<\tau^U=\min\{\tau,M^{X,Y}\}\big)\le \Prob_{\bf x}^{\mathrm{CAB}}\big(M^{U,V}<\tau\big)\wedge p.\]
\end{proposition}

Before proving this, we show how these propositions combined with Theorem~\ref{thm:normalApprox} yield Item~\ref{thm:grid1} of Proposition~\ref{P:grid}.
\begin{proof}[Proof of Proposition~\ref{P:grid} Item \ref{thm:grid1}]
The case $pn^2 \to 0$ follows immediately from Proposition~\ref{prop:small_p} (recall that $p$ depends on $n$), since the expected meeting time of two independent stationary walks is $\Theta(n^2)$. 

Suppose now that $pn^2 \to \infty$. Recall that $p$ is non-increasing in $n$, hence $\lim_n p$ exists. If $\lim_n p>0$ then the result follows from Corollary~\ref{cor:normalApprox}. Next, by Theorem~\ref{thm:normalApprox} and equation~\ref{eqn:varPsiineq1}, it suffices to show that if $n^{-2}\ll p \ll 1$ then 
\[
\frac1{n^2p\sigma^3}+\frac1{np\sigma^2}\sqrt{\sum_{x,y,u,v} \pi(x)\pi(y)P(x,u)P(y,v)h_0(u,v,x,y)}=o(1)
\]as $n\to\infty$. For the first term of the sum above, by Lemma~\ref{L:cyclevar}, we have $\sigma^2=\frac1{2\sqrt{ 2p}n}(1+o(1))$ and so $(n^2p\sigma^3)^{-1}=o(1)$ as $n\to\infty$. The second term in the sum is more complicated. We begin by noticing that

\begin{align}
    \sum_{x,y,u,v} \pi(x)\pi(y)P(x,u)P(y,v)h_0(u,v,x,y) =\E\left( h_0(U,V,X,Y)\right)
\end{align}
 where $X,Y\sim \pi$, $U$ is chosen according to $P(X,\cdot)$ and $V$ according to $P(Y,\cdot)$ all independently. Now observe that from the definition of $h$, 
\[h(u,v,x,y)=h(x,v,u,y)=h(u,y,x,v)=h(x,y,u,v),\]and thus the same is true also for $h_0$.
 Therefore we have
\begin{align*}
\mathbb{E}[h_0(U,V,X,Y)]&=\frac1{4n^2}\sum_{x,y\in C_n}\big[h_0(x_+,y_+,x,y)+h_0(x_+,y_-,x,y)\\
&\phantom{=\frac1{4n^2}\sum_{x,y\in C_n}\big[}+h_0(x_-,y_+,x,y)+h_0(x_-,y_-,x,y)\big]\\
&=\frac1{n^2}\sum_{x,y\in C_n}h_0(x_+,y_-,x,y),
\end{align*}
where $x_+$, $x_-$ for $x\in C_n$ refers to the element of $C_n$ neighbouring $x$ clockwise, counter-clockwise respectively. Notice also that $h_0(u,v,x,y)=h_0(y,x,v,u)$ and hence the above can be written as
\begin{align}
\begin{split}\mathbb{E}[h_0(U,V,X,Y)]
&=\frac1{n^2}\sum_{x\in C_n}\sum_{\substack{y\in C_n:\\1\le d(y,x)<\lfloor n/2\rfloor}}h_0(y,x,y_-,x_+)\\&+\frac1{n^2}\sum_{x\in C_n}\sum_{\substack{y\in C_n:\\d\lfloor n/2\rfloor\le d(y,x)\le n-3}}h_0(x_+,y_-,x,y)\\
&\phantom{=}+\frac1{n^2}\sum_{x\in C_n}\sum_{\substack{y\in C_n:\\n-2\le d(y,x)\le n}}h_0(x_+,y_-,x,y).\end{split}\label{e:Eh}
\end{align}
\underline{If $d(y,x)=n-2$} then $x_+=y_-$ and $h_0(x_+,y_-,x,y)=g_0(x,y,x_+)$ where $g$ is as in Lemma~\ref{L:gmart} and $g_0=g\vee 0$. To bound these terms we construct a discrete-time process $(M_t)_{t\in\mathbb{N}_0}$ defined in terms of the locations of the particles in the dual process as
\begin{align}
M_t=\left(1-\zeta^{d(X_t,(X_+)_t)}\right)\left(1-\zeta^{d((X_+)_t,Y_t)}\right)\indic{t<\tau_\mathrm{abs}}+\indic{t\ge\tau_\mathrm{abs}},\label{e:mart1}
\end{align}
where $\tau_\mathrm{abs}=\min\{\tau^x,\tau^y,\tau^{x_+}\}$. This process is a bounded martingale until time \[T:=\min\{\tau_\mathrm{abs},M^{x,y},M^{x,x_+},M^{y,x_+}\}\] with $\zeta=\frac{2+p-\sqrt{3p(4-p)}}{2(1-p)}=1-\sqrt{3p}+O(p)$ as $p\to0$ (details are given later for the similar martingale appearing in \eqref{e:mart2}). Then by the optional stopping theorem we have 
\begin{align*}
(1-\zeta)^2&=\mathbb{E}\left[\indic{T=\tau_\mathrm{abs}}+(1-\zeta^{d(X_T,(X_+)_T)})(1-\zeta^{d((X_+)_T,X_T)})\indic{T=M^{x,y}}\right]\\
&\ge \mathbb{E}\left[\indic{T=\tau_\mathrm{abs}}+(1-\zeta^{d(X_T,(X_+)_T)}-\zeta^{d((X_+)_T,X_T)})\indic{T=M^{x,y}}\right]\\
&\ge \mathbb{E}\left[\indic{T=\tau_\mathrm{abs}}+(1-\zeta-\zeta^{n-1})\indic{T=M^{x,y}}\right]\\
&=\mathbb{E}\left[\indic{T=\tau_\mathrm{abs}}+\sqrt{3p}(1+o(1))\indic{T=M^{x,y}}\right]\\
&\ge \sqrt{3p}(1+o(1))\left(\mathbb{P}(T=\tau_\mathrm{abs})+\mathbb{P}(T=M^{x,y})\right)\\
&\ge \sqrt{p/3}(1+o(1))\left(\mathbb{P}(T=\tau^{x_+})+\mathbb{P}(T=M^{x,y})\right)\\
&\ge \sqrt{p/3}(1+o(1))g(x,y,x_+)
\end{align*}
as $n\to\infty$.
We deduce that there exists a universal $c>0$ such that for all $n$ sufficiently large $g(x,y,x_+)\le c\sqrt p$. It follows that also $g_0(x,y,x_+)\le c\sqrt p$.

\underline{If $d(y,x)=n-1$} then $x=y_-$ and $y=x_+$ so $h(x_+,y_-,x,y)=\mathbb{P}(B(x)\neq B(x_+))-\mathbb{P}(B(x)\neq B(x_+))^2\le\mathbb{P}(B(x)\neq B(x_+))=\frac12\left(1-\mathbb{P}\left(M^{x,x_+}<\tau^x\wedge\tau^{x_+}\right)\right)$. We can now appeal to \eqref{eq:ERWk} to deduce that $\mathbb{P}\left(M^{x,x_+}<\tau^x\wedge\tau^{x_+}\right)=(\theta+\theta^{n-1})(1+\theta^n)^{-1}$ where $\theta=(1-\sqrt{p(2-p)})(1-p)^{-1}$, i.e.\! $\mathbb{P}\left(M^{x,x_+}<\tau^x\wedge\tau^{x_+}\right)=1-\sqrt{2p}+O(p)$ as $p\to0$, and so there exists a universal $c>0$ such that for all $n$ sufficiently large $h(x_+,y_-,x,y)\le c\sqrt p$. It follows that also $h_0(x_+,y_-,x,y)\le c\sqrt p$.

\underline{If $d(y,x)=n$} then $x=y$ and $h_0(x_+,y_-,x,y)=g_0(x_+,y_-,x)$. By the same argument as the $d(y,x)=n-2$ case, we have $g_0(x_+,y_-,x)\le c\sqrt p$ eventually.

\underline{If $1\le d(x,y)\le n-3$} (which coincides with $\{x,y,x_+,y_-\}$ being distinct), we can apply Corollary~\ref{C:cyclecov} followed by Propositions~\ref{P:cyclemain} and~\ref{P:cycle2ndprob} to each of these above terms. If further $d(y,x)<\lfloor n/2\rfloor$ then $d(x_+,y_-)\ge d(y,x)$ and so we obtain 
\begin{align*}
&h_0(y,x,y_-,x_+)\\&\le \Prob_{{\bf y}}^{\mathrm{CAB}}\big(M^{X_+,Y_-}<\tau\big)
+\Prob_{{\bf y}}^{\mathrm{CAB}}\big(M^{X,Y}<\tau^X=\min\{\tau,M^{X_+,Y_-}\}\big)\\
&\le Cd(x_+,y_-)^{-2}\left(e^{-C\sqrt p d(x_+,y_-)}+d(x_+,y_-)^{-1}\right)\\
&\phantom{\le}+\left\{Cd(y,x)^{-2}\left(e^{-C\sqrt p d(y,x)}+d(y,x)^{-1}\right)\right\}\wedge p,
\end{align*}
where ${\bf y}=(y_-,y,x,x_+)$ and $\tau=\min\{\tau^X,\tau^{X_+},\tau^Y,\tau^{Y_-},M^{X,X_+},M^{Y,Y_-}\}$.
Since $d(y,x)<\lfloor n/2\rfloor$ we deduce that for $n\ge 4$, $d(x_+,y_-)\ge n/4$, hence there exists a universal constant $\hat C>0$ such that for all $n$ sufficiently large, we have the bound
\begin{align}\label{e:dyxsmall}
h_0(y,x,y_-,x_+)&\le \frac{\hat C}{n^2}\left(e^{-\hat C\sqrt p n}+\frac1{n}\right)+\left\{Cd(y,x)^{-2}\left(e^{-C\sqrt p d(y,x)}+d(y,x)^{-1}\right)\right\}\wedge p.
\end{align}

If instead we have $d(y,x)\ge \lfloor n/2\rfloor$ (with $\{x,y,x_+,y_-\}$ still distinct) then $d(y,x)\ge d(x_+,y_-)$ and we similarly obtain for all $n$ sufficiently large 
\begin{align}
&\notag h_0(x_+,y_-,x,y)\\\notag&\le \Prob_{\bf x}^{\mathrm{CAB}}\big(M^{X,Y}<\tau\big)
+\Prob_{\bf x}^{\mathrm{CAB}}\big(M^{X_+,Y_-}<\tau\big)\wedge\Prob_{\bf x}^{\mathrm{CAB}}\big(M^{X_+,Y_-}<\tau^{X_+}=\min\{\tau,M^{X,Y}\}\big)\\
&\le Cd(y,x)^{-2}\left(e^{-C\sqrt p d(y,x)}+d(y,x)^{-1}\right)\notag\\
&\phantom{\le}+\left\{Cd(x_+,y_-)^{-2}\left(e^{-C\sqrt p d(x_+,y_-)}+d(x_+,y_-)^{-1}\right)\right\}\wedge p,\notag\\
&\le \frac{\hat C}{n^2}\left(e^{-\hat C\sqrt p n}+\frac1{n}\right)+\left\{Cd(x_+,y_-)^{-2}\left(e^{-C\sqrt p d(x_+,y_-)}+d(x_+,y_-)^{-1}\right)\right\}\wedge p\label{e:dyxlarge},
\end{align}
where ${\bf x}=(x,x_+,y,y_-)$ and $\tau=\min\{\tau^X,\tau^{X_+},\tau^Y,\tau^{Y_-},M^{X,X_+},M^{Y,Y_-}\}$.
Putting the bounds from \eqref{e:dyxsmall} and \eqref{e:dyxlarge} together with the bound of $c\sqrt p$ for the other cases into \eqref{e:Eh} we obtain that there exist constants $c_1,c_2,c_3,c_4>0$ such that for all $n$ sufficiently large
\begin{align*}
&\mathbb{E}[h_0(U,V,X,Y)]\\&\le\frac{c_1}{n}\sum_{d=1}^{\lfloor n/2\rfloor}\left[n^{-2}\left(e^{-c_1\sqrt p n}+n^{-1}\right)+\left(d^{-2}\left(e^{-c_1\sqrt p d}+d^{-1}\right)\right)\wedge p\right]+\frac{\sqrt p}{n}\\
&\le \frac{c_1}{n^2}+\frac{c_2}{n}\sum_{d=1}^{\lfloor(\log 1/p)/\sqrt p\rfloor}\left(d^{-2}e^{-c_1\sqrt p d}\right)\wedge p+\frac{c_2}{n}\sum_{d=\lfloor(\log 1/p)/\sqrt p\rfloor}^{\lfloor n/2\rfloor}d^{-3}\wedge p+\frac{\sqrt p}{n}\\
&\le \frac{c_1}{n^2}+\frac{c_3}{n}\sum_{d=1}^{\lfloor1/\sqrt p\rfloor}p+\frac{c_3}{n}\sum_{d=\lfloor1/\sqrt p\rfloor}^{\lfloor(\log 1/p)/\sqrt p\rfloor}d^{-2}+\frac{c_2}{n}\sum_{d=\lfloor(\log 1/p)/\sqrt p\rfloor}^{\lfloor n/2\rfloor}d^{-3}+\frac{\sqrt p}{n}\\
&\le \frac{c_1}{n^2}+\frac{c_3\sqrt p}{n}+\frac{c_4\sqrt p}{n}+\frac{c_4p}{n}+\frac{\sqrt p}{n}.
\end{align*}
Recalling that $\sigma^2=\frac1{2\sqrt{ 2p}n}(1+o(1))$ as $n\to\infty$, it follows that there exists a constant $c_5>0$ such that for all $n$ sufficiently large,
\[
\frac1{np\sigma^2}\sqrt{\mathbb{E}[h_0(U,V,X,Y)]}\le \frac{c_5}{\sqrt p}\sqrt{n^{-2}+\sqrt pn^{-1}}
\]and the right-hand side tends to 0 as $n\to\infty$ by the assumption that $p\gg n^{-2}$.
\end{proof}
\begin{proof}[Proof of Proposition~\ref{P:cycle2ndprob}]
The bound \[\Prob_{\bf x}^{\mathrm{CAB}}\big(M^{U,V}<\tau^U=\min\{\tau,M^{X,Y}\}\big)\le \Prob_{\bf x}^{\mathrm{CAB}}\big(M^{U,V}<\tau\big)\]
is immediate. To bound the probability on the left-hand side above by $p$, we shall use martingale optional stopping arguments. Before defining the appropriate martingales, we introduce a process on the line $\mathbb{Z}$. We define a process WAB ({\bf w}alks with {\bf ab}sorption) as a process $({\bf X}^*_t)_{t\in\mathbb{N}_0}=(X_t^*,U_t^*,Y_t^*)_{t\in\mathbb{N}_0}$ taking values in $(\mathbb{Z}\cup \Delta)^3$ where $\Delta$ denotes the cemetery state. At each step, one of the three particles is chosen uniformly. The chosen particle jumps to the cemetery state with probability $p$, jumps down 1 with probability $(1-p)/2$ and otherwise jumps up 1. The process terminates at the first time two particles meet or any particle jumps to the cemetery state.


Now we construct a discrete-time process $(M_t)_{t\in\mathbb{N}_0}$ defined in terms of the locations of the particles in the WAB process. This process is almost identical to the one defined in \eqref{e:mart1} except the underlying process here is a WAB process on the line and the prior martingale is the dual process on the cycle. Specifically, for each $t\in\mathbb{N}_0$ we set
\begin{align}\label{e:mart2}
M_t=\left(1-\zeta^{U_t^*-X_t^*}\right)\left(1-\zeta^{Y_t^*-U_t^*}\right)\indic{t<\tau_\mathrm{abs}}+\indic{t\ge\tau_\mathrm{abs}},
\end{align}
for $\zeta=\zeta(p)\in(0,1)$ to be determined and where $\tau_\mathrm{abs}=\min\{\tau^{X^*},\tau^{U^*},\tau^{Y^*}\}$ (with this notation again referring to absorption times). The choice of $\zeta$ will ensure $(M_t)$ is a martingale.
Let $\mathcal{F}_t$ be the filtration generated by the trajectories of all particles in WAB up to time $t$. Then for each $t\in\mathbb{N}_0$,
\begin{align*}
&\mathbb{E}[M_{t+1}\mid \mathcal{F}_t]\\=&\mathbb{E}\left[\left(1-\zeta^{U_{t+1}^*-X_{t+1}^*}\right)\left(1-\zeta^{Y_{t+1}^*-U_{t+1}^*}\right)\indic{t+1<\tau_\mathrm{abs}}+\indic{t+1\ge\tau_\mathrm{abs}}\mid\mathcal{F}_t\right]\\
=&\indic{t\ge\tau_\mathrm{abs}} +p\indic{t<\tau_\mathrm{abs}}\\&+\frac{1-p}{6}\indic{t<\tau_\mathrm{abs}}\bigg\{\left(1-\zeta^{U_{t}^*-X_{t}^*-1}\right)\left(1-\zeta^{Y_{t}^*-U_{t}^*}\right)
+\left(1-\zeta^{U_{t}^*-X_{t}^*+1}\right)\left(1-\zeta^{Y_{t}^*-U_{t}^*}\right)\bigg\}\\
&+\frac{1-p}{6}\indic{t<\tau_\mathrm{abs}}\bigg\{\left(1-\zeta^{U_{t}^*-X_{t}^*+1}\right)\left(1-\zeta^{Y_{t}^*-U_{t}^*-1}\right)\\
&\phantom{+\frac{1-p}{6}\indic{t<\tau_\mathrm{abs}}\bigg\{}+\left(1-\zeta^{U_{t}^*-X_{t}^*-1}\right)\left(1-\zeta^{Y_{t}^*-U_{t}^*+1}\right)\bigg\} \\
&+\frac{1-p}{6}\indic{t<\tau_\mathrm{abs}}\bigg\{\left(1-\zeta^{U_{t}^*-X_{t}^*}\right)\left(1-\zeta^{Y_{t}^*-U_{t}^*+1}\right)+\left(1-\zeta^{U_{t}^*-X_{t}^*}\right)\left(1-\zeta^{Y_{t}^*-U_{t}^*-1}\right) \bigg\}\\
&=\indic{t\ge\tau_\mathrm{abs}}+\indic{t<\tau_\mathrm{abs}}\Bigg[1+\frac{1-p}{3}\left(1+\zeta+1/\zeta\right)\left(\zeta^{Y_{t}^*-X_{t}^*}-\zeta^{U_{t}^*-X_{t}^*}-\zeta^{Y_{t}^*-U_{t}^*}\right)\Bigg].
\end{align*}
Thus for $(M_t)$ to be a martingale it suffices that $1+\zeta+1/\zeta=3/(1-p)$, i.e.\! we take 
\begin{align}\label{e:zetabound}
\zeta=\frac{2+p-\sqrt{3p(4-p)}}{2(1-p)}\ge1-\sqrt{3p}.
\end{align}
 Let $T^*:=\tau_\mathrm{abs}\wedge M^{X^*,U^*}\wedge M^{Y^*,U^*}$ and observe that $|M_t|\le 1$ for all $t\le T^*$. Thus by the optional stopping theorem, \begin{align}\label{e:OST1}
\mathbb{E}[(1-\zeta^{U^*_0-X^*_0})(1-\zeta^{Y^*_0-U^*_0})]=\mathbb{E}[M_0]=\mathbb{E}[M_{T^*}]=\mathbb{P}^\mathrm{WAB}(\tau_\mathrm{abs}<M^{X^*,U^*}\wedge M^{Y^*,U^*}).
\end{align}

We turn to bounding the probability of interest. 
It is convenient to couple the CAB process to a process denoted CAB$'$ which is initialised from the same state and which evolves in the same way except that there can be no absorptions until the first meeting time (one way to construct is to take the CAB process and set $p=0$ until the first meeting occurs). The coupling is successful (i.e.\! the two processes agree for all times) for CAB processes which satisfy that the first absorption is after the first meeting. Thus
\begin{align*}
&\Prob_{\bf x}^{\mathrm{CAB}}\big(M^{U,V}<\tau^U=\min\{\tau,M^{X,Y}\}\big)\\&\le\Prob_{\bf x}^{\mathrm{CAB}'}\big(M^{U,V}<\tau^U=\min\{\tau,M^{X,Y}\}\big)\\
&= \Prob_{\bf x}^{\mathrm{CAB}'}\big(M^{U,V}< \min\{M^{X,U},M^{X,Y},M^{U,Y}\},\, \tau^U=\min\{\tau,M^{X,Y}\}\big).
\end{align*}

Now let $\mathcal{F}^{\mathrm{CAB}'}_t$ be the sigma algebra generated by the CAB$'$ process up to time $t$ and $R:=\min\{\tau^X,\tau^U,\tau^Y\}$. Then
\begin{align}
&\Prob_{\bf x}^{\mathrm{CAB}}\big(M^{U,V}<\tau^U=\min\{\tau,M^{X,Y}\}\big)\notag\\
&\le\mathbb{E}^{\mathrm{CAB}'}_{\bf x}\left[\mathbb{E}\left[\indic{M^{U,V}<\min\{M^{X,U},M^{X,Y},M^{V,Y}\}}\indic{\tau^U=\min\{\tau,M^{X,Y}\}}\mid\mathcal{F}^{\mathrm{CAB}'}_{M^{U,V}},\,M^{U,V}\right]\right]\notag\\
&\le\frac13\mathbb{E}^{\mathrm{CAB}'}_{\bf x}\Big[\indic{M^{U,V}<\min\{M^{X,U},M^{X,Y},M^{V,Y}\}}\notag\\&\phantom{\le\frac13\mathbb{E}^{\mathrm{CAB}'}_{\bf x}\Big[}\cdot\mathbb{E}\left[\indic{R<\min\{M^{X,U},M^{X,Y},M^{V,Y}\}}\mid\mathcal{F}^{\mathrm{CAB}'}_{M^{U,V}},\,M^{U,V}\right]\Big].\label{e:CAB1}
\end{align}
The random variable 
\[
\indic{M^{U,V}<\min\{M^{X,U},M^{X,Y},M^{V,Y}\}}\mathbb{E}\left[\indic{R<\min\{M^{X,U},M^{X,Y},M^{V,Y}\}}\mid\mathcal{F}^{\mathrm{CAB}'}_{M^{U,V}},\,M^{U,V}\right]\]
is clearly only non-zero for trajectories satisfying $M^{U,V}<\min\{M^{X,U},M^{X,Y},M^{V,Y}\}$. For these, we have
by the memoryless property ($R$ is geometrically distributed) and the strong Markov property at time $S:=M^{U,V}$, almost surely
\begin{align*}
&\mathbb{E}\left[\indic{R<\min\{M^{X,U},M^{X,Y},M^{V,Y}\}}\mid\mathcal{F}^{\mathrm{CAB}'}_{S},\,S\right]\\
&\le \Prob_{({ X}_{S},{ U}_{S},{ V}_{S},{Y}_{S})}^{\mathrm{CAB}'}(R<\min\{M^{X,U},M^{X,Y},M^{V,Y}\}).
\end{align*}
The CAB$'$ process in the last probability is initialised from the random state $({ X}_{S},{ U}_{S},{ V}_{S},{Y}_{S})$, and as such $U$ and $V$ will remain together for all later times; thus we can write the preceding probability as (we change $M^{V,Y}$ to $M^{U,Y}$)
\[
\Prob_{({ X}_{S},{ U}_{S},{ V}_{S},{Y}_{S})}^{\mathrm{CAB}}(R<\min\{M^{X,U},M^{X,Y},M^{U,Y}\}),
\]
where $(X_S,U_S,V_S,Y_S)$ is the time $S$ state of a CAB$'$ process (from time $S$ onwards the process behaves as a CAB process hence the above probability measure is $\Prob^\mathrm{CAB}$). 

This CAB can be coupled with a WAB process started from $(-d(X_{S},U_{S}),0,d(U_{S},Y_{S}))$ such that a jump clockwise/anti-clockwise in CAB corresponds to a jump up/down in WAB, and jumps to cemetery states coincide in the two processes. 

This coupling gives that, almost surely
\begin{align*}
&\Prob^{\mathrm{CAB}}_{({ X}_{S},{ U}_{S},{ V}_{S},{Y}_{S})}\left(R<\min\{M^{X,U},M^{X,Y},M^{U,Y}\}\right)\\&\le \Prob^{\mathrm{WAB}}_{(-d(X_{S},U_{S}),0,d(U_{S},Y_{S}))}\left(\tau_\mathrm{abs}<\min\{M^{X^*,U^*},M^{Y^*,U^*}\}\right),
\end{align*}
and combining this with \eqref{e:OST1} we obtain
\begin{align*}
&\Prob^{\mathrm{CAB}}_{({ X}_{S},{ U}_{S},{ V}_{S},{Y}_{S})}\left(R<\min\{M^{X,U},M^{X,Y},M^{U,Y}\}\right)
\\&\le \mathbb{E}^\mathrm{WAB}\left[(1-\zeta^{d(X_S,U_S)})(1-\zeta^{d(U_S,Y_S)})\mid X_S,U_S,Y_S\right]\\
&\le 3p\,d(X_S,U_S)\,d(U_S,Y_S)
\end{align*}
almost surely, where we have used the bound \eqref{e:zetabound} in the final inequality. 

Plugging this into \eqref{e:CAB1} we obtain
\begin{align}
&\Prob_{\bf x}^{\mathrm{CAB}}\big(M^{U,V}<\tau^U=\min\{\tau,M^{X,Y}\}\big)\notag\\
&\le p \,\mathbb{E}^{\mathrm{CAB}'}_{\bf x}\left[\indic{M^{U,V}<\min\{M^{X,U},M^{X,Y},M^{V,Y}\}}d(X_{M^{U,V}},U_{M^{U,V}})\,d(U_{M^{U,V}},Y_{M^{U,V}})\right].\label{e:need2ndmart}
\end{align}
Now the reason for introducing the process CAB$'$ becomes apparent: we know that, on event $\{M^{U,V}<\min\{M^{X,U},M^{X,Y},M^{V,Y}\}\}$, until time $M^{U,V}$ there are no absorptions (by definition of this process). Also on this event we have (by the nature of the cycle) that for all $t\le M^{U,V}$, $d(X_{M^{U,V}},U_{M^{U,V}})\,d(U_{M^{U,V}},Y_{M^{U,V}})<n^2$. Thus we can bound the above expectation by coupling to a process of simple random walks on the line, $\RW(4)$, defined as follows: $\RW(4)$ takes values in $\mathbb{Z}^4$ and at each step a uniformly chosen particle jumps up or down equally likely. Let $(A_t,B_t,C_t,D_t)$ be the time-$t$ state of $\RW(4)$. The coupling is again simple: clockwise/anti-clockwise movements of particles correspond to jumps up/down. We obtain the inequality
\begin{align*}
&\mathbb{E}^{\mathrm{CAB}'}_{\bf x}\left[\indic{M^{U,V}<\min\{M^{X,U},M^{X,Y},M^{V,Y}\}}d(X_{M^{U,V}},U_{M^{U,V}})\,d(U_{M^{U,V}},Y_{M^{U,V}})\right]\\
&\le\mathbb{E}^{\mathrm{RW}(4)}_{(-d(x,u),0,d(u,v),d(u,y))}\Big[\indic{M^{B,C}<M^{A,B}\wedge M^{C,D}}\\
&\phantom{\le\mathbb{E}^{\mathrm{RW}(4)}_{(-d(x,u),0,d(u,v),d(u,y))}\Big[}\cdot\left\{(B_{M^{B,C}}-A_{M^{B,C}})(D_{M^{B,C}}-B_{M^{B,C}})\wedge n^2\right\}\Big]\\
&=\mathbb{E}^{\mathrm{RW}(4)}_{(-d(x,u),0,d(u,v),d(u,y))}\Big[\indic{M^{B,C}<M^{A,B}\wedge M^{C,D}}\\
&\phantom{=\mathbb{E}^{\mathrm{RW}(4)}_{(-d(x,u),0,d(u,v),d(u,y))}\Big[}\cdot\left\{(B_{M^{B,C}}-A_{M^{B,C}})(D_{M^{B,C}}-C_{M^{B,C}})\wedge n^2\right\}\Big]\\\\
&=\mathbb{E}^{\mathrm{RW}(4)}_{(-d(x,u),0,d(u,v),d(u,y))}[N_M],
\end{align*}
where the process, $(N_t)_{t\in\mathbb{N}_0}$, defined as 
\[
N_t=(B_t-A_t)(D_t-C_t)\wedge n^2
\]
is a bounded submartingale until the first meeting time $M$ of any two particles in $\RW(4)$. Thus by the optional stopping theorem,
\begin{align*}
1&=d(x,u)\left(d(u,y)-d(u,v)\right)\\&=\mathbb{E}^{\mathrm{RW}(4)}_{(-d(x,u),0,d(u,v),d(u,y))}[N_0]\ge\mathbb{E}^{\mathrm{RW}(4)}_{(-d(x,u),0,d(u,v),d(u,y))}[N_M].
\end{align*}
Combining with \eqref{e:need2ndmart} we obtain
\[
\Prob_{\bf x}^{\mathrm{CAB}}\big(M^{U,V}<\tau^U=\min\{\tau,M^{X,Y}\}\big)\le p,
\]
which completes the proof.
\end{proof}

We now turn to proving Proposition~\ref{P:cyclemain}. The first step is a preliminary lemma on the decay rate of an expectation involving the hitting time of zero by a discrete-time random walk on $\mathbb{Z}$. To this end, we write $\mathbb{E}^\mathrm{RW}$ and $\mathbb{P}^\mathrm{RW}$ for the expectation and probability measure associated with such a random walk and $\tau_0^X$ for the hitting time of zero by process $X$.

\begin{lemma}\label{L:LCLT}There exists a constant $c>0$ such that for each $p\in[0,1)$,
\[ \mathbb{E}^\mathrm{RW}_{\lfloor d/8\rfloor}[(\tau^X_{0})^{-1/2}(1-p)^{\tau_0^X/3}]\le \begin{cases}
cd^{-1}&\mbox{ if }p=0,\\
c\left(d^{-1}e^{-c d\sqrt p}+d^{-2}\right)&\mbox{ if }p\in(0,1).
\end{cases}\]
\end{lemma}
\begin{proof}We first decompose over the hitting time of 0:
\begin{align*}
\mathbb{E}^\mathrm{RW}_{\lfloor d/8\rfloor}[(\tau^X_{0})^{-1/2}(1-p)^{\tau_0^X/3}]=\sum_{i=\lfloor d/8\rfloor}^{\infty} i^{-1/2}(1-p)^{i/3}\Prob_{\lfloor d/8\rfloor}^\mathrm{RW}[\tau_0^X=i].
\end{align*}
Next, by the hitting time theorem (Theorem~\ref{thm:hittingtime}) for each $i\ge \lfloor d/8\rfloor$,
\[
\Prob_{\lfloor d/8\rfloor}^\mathrm{RW}(\tau_0^X=i)=\frac1{i}\left\lfloor\frac{d}{8}\right\rfloor\Prob_{\lfloor d/8\rfloor}^\mathrm{RW}(X_i=0).
\] 
We can now apply the local central limit theorem which gives the existence of constants $c_1,c_2>0$ such that, uniformly in $i\ge \lfloor d/8\rfloor$, 
\begin{align*}
\Prob_{\lfloor d/8\rfloor}^\mathrm{RW}(X_i=0)\le c_1i^{-1/2}\left(\exp(-c_2d^2/i)+d^{-2})\right).
\end{align*}
Using this gives the existence of constants $c_2,c_3>0$ such that
\begin{align*}
\sum_{i=\lfloor d/8\rfloor}^{\infty} i^{-1/2}(1-p)^{i/3}\Prob_{\lfloor d/8\rfloor}^\mathrm{RW}[\tau_0^X=i]&\le c_3d\sum_{i=\lfloor d/8\rfloor}^{\infty} i^{-2}(1-p)^{i/3}\left(\exp(-c_2d^2/i)+d^{-2})\right)
\end{align*}
We treat separately the case $p=0$. In this case we have the bound
\begin{align*}
\mathbb{E}^\mathrm{RW}_{\lfloor d/8\rfloor}[(\tau^X_{0})^{-1/2}]\le c_3d\sum_{i=\lfloor d/8\rfloor}^{\infty} i^{-2}\left(\exp(-c_2d^2/i)+d^{-2})\right)\le c_4 d^{-1},
\end{align*}
for some $c_4>0$. For $p\in(0,1)$, we have the bound
\begin{align*}
&\mathbb{E}^\mathrm{RW}_{\lfloor d/8\rfloor}[(\tau^X_{0})^{-1/2}(1-p)^{\tau_0^X/3}]\\&\le 
c_3d\sum_{i=\lfloor d/8\rfloor}^{\infty} i^{-2}e^{-ip/3}\exp(-c_2d^2/i)+c_5d^{-2}\\
&\le 
c_3d\sum_{i=\lfloor d/8\rfloor}^{d/\sqrt p} i^{-2}\exp(-c_2d^2/i)+c_3d\sum_{i=d/\sqrt p}^{\infty} i^{-2}e^{-ip/3}+c_5d^{-2}\\
&\le c_6\left(d^{-1}e^{-c_6d\sqrt p}+d^{-2}\right),
\end{align*} 
for some $c_5,c_6>0$.
\end{proof}

Given $n\in\mathbb{N}$ and $p\in(0,1)$ (possibly depending on $n$), we define two new processes, denoted WAG and WAIG. Process WAG ({\bf w}alks with {\bf a}bsorption and {\bf g}host) is a  process $({\bf X}_t)_{t\in\mathbb{N}_0}=(X_t,U_t,V_t,Y_t)_{t\in\mathbb{N}_0}$ taking values in $(\Z\cup \Delta)^4$ where $\Delta$ denotes the cemetery state. This process starts at time 0 from a state ${\bf x}=(x,u,v,y)$ with $u\le x\le0\le y\le v$ and we say that all 4 particles are \emph{alive} unless $v-u=n$ in which case all particles except $V$ are alive, with $V$ being a \emph{ghost}. At each time $t\in\mathbb{N}$ we choose uniformly  from the alive particles and this chosen particle with probability $p$ gets absorbed into the cemetery state $\Delta$ and dies (is no longer alive), otherwise the particle chooses to jump up by one or down by one with equal probability. If this jump results in $V-U=n$, we say that $V$ becomes a ghost (and is no longer alive). If $V$ was already a ghost and the particle chosen is $U$, then $V$ evolves in the same way as $U$, i.e.\! if $U$ is absorbed into $\Delta$ then so is $V$, if $U$ jumps up/down then $V$ jumps up/down. The process terminates (no further evolution occurs) if $U$ and $X$ meet, or $Y$ and $V$ meet,  either of $X$ and $Y$ hits 0, or any particle is absorbed into the cemetery state.

Process WAIG ({\bf w}alks with {\bf a}bsorption and {\bf i}ndependent {\bf g}host) is almost identical to WAG. The only difference is that if $V$ is a ghost and $U$ is chosen and not absorbed into $\Delta$, then $V$ jumps up/down independently of $U$, rather than copying its move.

In each of these systems we use notation $\tau^A$ to denote the absorption time into $\Delta$ of particle $A\in\{X,U,V,Y\}$, $\tau_0^A$ the time for particle $A$ to hit 0, $M^{A,B}$ to denote meeting time of particles $A,B\in\{X,U,V,Y\}$, and $G^V$ to denote the time that $V$ becomes a ghost
(any of these times is infinite if the process terminates prior to their occurrence).
The following lemma relates the probability that $X$ or $Y$ hits 0 before time $\tau:=M^{X,U}\wedge M^{V,Y}\wedge \tau^Y$ in these two processes. 

\begin{lemma}\label{L:WAG2WAIG}For each ${\bf x}=(x,u,v,y)\in \Z^4$ with $u\le x\le y\le v$,
\[\Prob_{{\bf x}}^{\mathrm{WAG}}(\tau^{X}_0\wedge \tau_0^Y <M^{X,U}\wedge M^{V,Y}\wedge \tau^Y) \leq \Prob
_{{\bf x}}^{\mathrm{WAIG}}(\tau^{ X}_0\wedge \tau_0^Y < M^{X,U}\wedge M^{V,Y}\wedge \tau^Y).\]
\end{lemma}
\begin{proof}
We first note that $(\Prob_{{\bf X}_t}^{\mathrm{WAG}}(\tau_0^{X}\wedge \tau_0^Y<\tau))_{t\in\mathbb{N}_0}$ is a bounded martingale until time $T:=\tau_0^X\wedge \tau_0^Y\wedge \tau$ and so by the optional stopping theorem at time $S:=T\wedge G^{V}$,
\begin{align}
\notag\Prob_{{\bf x}}^{\mathrm{WAG}}(\tau^X_0\wedge \tau_0^Y <\tau)&=\E^{\mathrm{WAG}}[\Prob^{\mathrm{WAG}}_{{\bf X}_{S}}(\tau_0^X\wedge \tau_0^Y<\tau)]\\
&=\E^{\mathrm{WAG}}\left[\indic{\tau_0^X\wedge \tau_0^Y<\tau\wedge G^V}+\indic{G^V<T}\Prob^{\mathrm{WAG}}_{{\bf X}_{G^V}}(\tau_0^X\wedge \tau_0^Y<\tau)\right].\label{eq:WAG}
\end{align}
By the same reasoning we also have
\begin{align}
\Prob_{{\bf x}}^{\mathrm{WAIG}}(\tau^X_0\wedge \tau_0^Y <\tau)&=\E^{\mathrm{WAIG}}\left[\indic{\tau_0^X\wedge \tau_0^Y<\tau\wedge G^V}+\indic{G^V<T}\Prob^{\mathrm{WAIG}}_{{\bf X}_{G^V}}(\tau_0^X\wedge \tau_0^Y<\tau)\right].\label{eq:WAIG}
\end{align}
Notice that at time $G^V$, we know that $V-U=n$. So consider a state ${\bf a}=(a,b,n+b,c)$. We will show that
\begin{align}\label{eq:WAGWAIG}
\Prob^{\mathrm{WAG}}_{\bf a}(\tau_0^X\wedge \tau_0^Y<\tau)\le \Prob^{\mathrm{WAIG}}_{\bf a}(\tau_0^X\wedge \tau_0^Y<\tau).
\end{align}
Once we have this inequality we can complete the proof by taking a coupling of WAG and WAIG: we can couple them so that the absorption times, $G^V$ and the trajectories of all particles are the same in both processes up to time $G^V$. It follows that under this coupling if each of the events   $\{\tau_0^X\wedge \tau_0^Y<\tau\wedge G^V\}$ and $\{G^V<T\}$ hold in WAG, they hold in WAIG and vice-versa. Hence the indicators in equations \eqref{eq:WAG} and \eqref{eq:WAIG} are equal.

It remains to show \eqref{eq:WAGWAIG}.
Note that starting WAG from ${\bf a}$ means that $V$ is immediately a ghost. To show \eqref{eq:WAGWAIG} we will make use of the FKG inequality for partially ordered sets. In process WAG, we denote by $(Q_i)_{i\in\mathbb{N}}$ a sequence of Rademacher random variables ($\pm 1$ with probability 1/2 each), which will be used (in the obvious way) to generate the movement of $U$ (and hence, since $V$ starts as a ghost, of $V$). To highlight the dependence of $U$ and $V$ on ${\bf Q}$ we will sometimes write these processes as $U({\bf Q})$ and $V({\bf Q})$ (note that if $U$ has jumped $k$ times by time $t$ then $U_t({\bf Q})=u+\sum_{i=1}^k Q_i$).

Let  $J^U$ be the sequence of times at which particle $U$ is chosen to evolve (jump or be absorbed) and 
$W = ((X_t,Y_t)_{t\in\mathbb{N}_0},\tau^{U},\tau^{V},J^U)$ and note that $\tau^X$ and $\tau^Y$ are measurable with respect to the sigma algebra generated by the trajectories of $X$ and $Y$. For ${\bf z}\in \{-1,+1\}^\mathbb{N}$, consider the functions
\begin{align*}
f({\bf z})&:=\indic{\tau_0^X\wedge \tau_0^Y<M^{X,U({\bf z})}\wedge \tau^X\wedge \tau^{U({\bf z})}}\\
g({\bf z})&:=\indic{\tau_0^X\wedge \tau_0^Y<M^{V({\bf z}),Y}\wedge \tau^{V({\bf z})}\wedge \tau^Y}.
\end{align*}

For each $m\in\mathbb{N}$, we define a partial order on the sample space $\Omega_m=\{-1,+1\}^m$ of $(Q_i)_{1\le i\le m}$. For $\omega_1,\omega_2\in\Omega_m$, we say that $\omega_1\le \omega_2$ iff $\omega_2$ can be obtained from $\omega_1$ by flipping some number of $-1$s in $\omega_1$ to $+1$s. Writing ${\bf 1}$ for an infinite vector of 1s and ${\bf Q}\in \Omega_m$, we observe that $f(({\bf Q},{\bf 1}))$ is decreasing with respect to this partial order. Conversely, $g_m$ is increasing. Let $\hat \mu_m$ be the measure on $\{-1,+1\}^\mathbb{N}$ which assigns non-zero and equal mass to vectors ${\bf z}$ with ${\bf z}(i)=1$ for all $i>m$, and let $\mu_m$ be $\hat\mu_m$ conditioned on $W$. By the FKG inequality we obtain that for each $m\in\mathbb{N}$,
\[
\mu_m(fg)\le\mu_m(f)\mu_m(g).
\] Sending $m\to\infty$\footnote{the limits exist since $\mu_m(fg)$, $\mu_m(f)$ and $\mu_m(g)$ remain unchanged for all $m\ge \tau_0^X\wedge\tau_0^Y$} we deduce that 
\begin{align*}
&\Prob^{\mathrm{WAG}}_{\bf a}(\tau_0^X\wedge \tau_0^Y<\tau\mid W)\\&\le \Prob^{\mathrm{WAG}}_{\bf a}(\tau_0^X\wedge \tau_0^Y<M^{X,U}\wedge \tau^X\wedge \tau^{U}\mid W)\, \Prob^{\mathrm{WAG}}_{\bf a}(\tau_0^X\wedge \tau_0^Y<M^{V,Y}\wedge \tau^V\wedge \tau^{Y}\mid W)\\
&=\Prob^{\mathrm{WAIG}}_{\bf a}(\tau_0^X\wedge \tau_0^Y<M^{X,U}\wedge \tau^X\wedge \tau^{U}\mid W)\, \Prob^{\mathrm{WAIG}}_{\bf a}(\tau_0^X\wedge \tau_0^Y<M^{V,Y}\wedge \tau^V\wedge \tau^{Y}\mid W)\\
&=\Prob^{\mathrm{WAIG}}_{\bf a}(\tau_0^X\wedge \tau_0^Y<\tau\mid W),
\end{align*}
where the second equality comes from the fact that given $W$, $U$ and $V$ move independently under WAIG. Taking an expectation gives \eqref{eq:WAGWAIG} and completes the proof.
\end{proof}

Now we bound the probability that, under WAIG, at least one of particles $X$ and $Y$ hits 0 before time $\tau$.
\begin{lemma}\label{L:WAIGbound}Suppose $d\ge 4$, $u,x,y,v$ are integers satisfying $u\le x\le -\frac{d}{2}+1\le \frac{d}{2}-\frac12\le y\le v$ and set ${\bf x}=(x,u,v,y)$. Then there exists a constant $c>0$ (not depending on $d$) such that 
\[\Prob
_{{\bf x}}^{\mathrm{WAIG}}(\tau^{ X}_0\wedge \tau_0^Y < \tau)\le cd^{-2}\left(e^{- c d\sqrt p}+d^{-1}\right).\]
\end{lemma}
\begin{proof}
We introduce a process denoted WMIG ({\bf w}alks with {\bf m}arking and {\bf i}ndependent {\bf g}host). This process is similar to WAIG except no particles are absorbed (when a particle is chosen it always jumps) and the process does not terminate when particles meet (and the dynamics of particles does not change when any two meet) or when $X$ or $Y$ hit 0. In addition, each time particle $Y$ is chosen,  with probability $p$ it is marked (if not already) -- this does not affect its movement at this time or afterwards. We denote by $\tau^Y_m$ the time at which $Y$ is marked.

 First we note that we can couple a WAIG process with a WMIG process if they start from the same configuration up until the time that the WAIG process terminates so that until this time particles move in the same way in the two processes. Since prior to time $\tau^X_0\wedge \tau_0^Y \wedge\tau$ the WAIG process does not terminate, we deduce that 
 \begin{align}\notag
&\Prob_{{\bf x}}^{\mathrm{WAIG}}(\tau^{ X}_0\wedge \tau_0^Y < \tau)\\\notag
&\le \Prob_{{\bf x}}^{\mathrm{WAIG}}(\tau^{ X}_0\wedge \tau_0^Y < M^{X,U}\wedge M^{V,Y}\wedge \tau^Y)
\\\notag&=\Prob
_{{\bf x}}^{\mathrm{WMIG}}(\tau^{ X}_0\wedge \tau_0^Y < M^{X,U}\wedge M^{V,Y}\wedge \tau^Y_m)\\&\le \Prob
_{{\bf x}}^{\mathrm{WMIG}}(\tau^{ X}_0 < M^{X,U}\wedge M^{V,Y}\wedge \tau^Y_m)+\Prob
_{{\bf x}}^{\mathrm{WMIG}}(\tau^{ Y}_0 < M^{X,U}\wedge M^{V,Y}\wedge \tau^Y_m).\label{eq:WAIG1}
 \end{align}
We now consider bounding $\Prob
_{{\bf x}}^{\mathrm{WMIG}}(\tau^{ X}_0 < M^{X,U}\wedge M^{V,Y}\wedge \tau^Y)$ (the other term can be bounded in the same way).

Given that $X_0\le -\frac{d}{2}-1$, in order for $X$ to hit 0 it must hit $-\lfloor d/4\rfloor$, a time we denote $\tau^X_{-\lfloor d/4\rfloor}$. We set $ T^X:=\tau^X_0-\tau^X_{-\lfloor d/4\rfloor}$ be the time it takes for $X$ to hit 0 from $-\lfloor d/4\rfloor$. By the strong Markov property $\tau^X_{-\lfloor d/4\rfloor}$ and $T^X$ are independent. We have
\begin{align}\label{eq:WAIG2}
&\Prob
_{{\bf x}}^{\mathrm{WMIG}}(\tau^{ X}_0 < M^{X,U}\wedge M^{V,Y}\wedge \tau^Y)\le \Prob
_{{\bf x}}^{\mathrm{WMIG}}(\tau^{ X}_{-\lfloor d/4\rfloor} < M^{X,U},\, T^X< M^{V,Y}\wedge \tau^Y)
\end{align}

The two events appearing in the probability on the right-hand side above are in fact independent (which is the reason for introducing the WMIG process). We make this clear with the following argument, which also gives us a way to bound this probability.

Let $C^{X,U}$ denote the number of times during time interval $[0,\tau^X_{-\lfloor d/4\rfloor})$ that process WMIG chooses either $X$ or $U$ to evolve. Let $C^{V,Y}$ denote the number of times during time interval $[\tau^X_{-\lfloor d/4\rfloor},\tau^X_0)$ that  $V$ or $Y$ jump under WMIG (if $V$ becomes a ghost it can no longer be chosen but still jumps if $U$ is chosen). Note that these two random variables are independent (they correspond to disjoint time intervals). Given $\tau^X_{-\lfloor d/4\rfloor}$, we have $C^{X,U}$ stochastically dominates a $\mathrm{Bin}(\tau^X_{-\lfloor d/4\rfloor},1/2)$ random variable (the stochastic domination comes from the fact that once $V$ becomes a ghost, we choose either $X$ or $U$ at each step with probability $2/3$). Similarly, given $T^X$, we have $C^{V,Y}$ stochastically dominates a $\mathrm{Bin}(T^X,1/2)$ random variable (once $V$ becomes a ghost, we choose either $Y$ or $U$ (choosing $U$ means $V$ also moves) at each step with probability $2/3$). 

For each $t\in\mathbb{N}_0$, for process WMIG set $d^1_t=X_t-U_t$ and $d^2_t=V_t-Y_t$ and for each $i\in\{1,2\}$ let $(D^i_t)_{t\in\mathbb{N}_0}$  be the process $(d^i_t)_{t\in\mathbb{N}_0}$ observed only when $d^i_t$ updates (i.e.\! the value changes). It is immediate that processes $(D^1_t)_{t\in\mathbb{N}_0}$ and $(D^2_t)_{t\in\mathbb{N}_0}$  are independent simple random walks on $\mathbb{Z}$ started from state 1. For each $i\in\{1,2\}$, let $\tau^{D^i}_0=\min\{t\ge0:\,D^i_t=0\}$ and let $\tau^D$ be the number of updates of $d^2_t$ until $Y$ is marked under WMIG so that $\tau^D\sim\mathrm{Geom}(1/2)$.

Then we have that under WMIG,\begin{align}\label{eq:WAIG3}\{\tau^X_{-\lfloor d/4\rfloor}<M^{X,U}\}=\{C^{X,U}<\tau^{D^1}_0\},\qquad \{T^X<M^{V,Y}\wedge\tau^Y\}=\{C^{V,Y}<\tau^{D^2}_0\wedge \tau^D\},\end{align}
and it is now clear that these events are independent. 

It is standard that there exists $c_1>0$ such that $\Prob^\mathrm{RW}_1(\tau_0^{D^1}>t)\sim c_1 t^{-1/2}$ as $t\to\infty$ (the superscript RW emphasises that $(D_t^1)_t$ is a discrete-time random walk on $\mathbb{Z}$). Hence there exists $c_2>0$ such that $\Prob_{\bf x}^\mathrm{WMIG}(C^{X,U}<\tau_0^{D^1})\le  c_2 \,\mathbb{E}_{\bf x}^\mathrm{WMIG}[(C^{X,U})^{-1/2}]$. By the maximal Azuma's inequality, there exists $c_3>0$ such that $\tau^X_{-\lfloor d/4\rfloor}$ is at least $d^2/(c_3\log d)$ with probability at least $1-d^{-3}$. Further, by a Chernoff bound, there exists $c_4>0$ such that given $\tau^X_{-\lfloor d/4\rfloor}$, $C^{X,U}$ is at least $\frac13 \tau_{-\lfloor d/4\rfloor}^X$ with probability at least $e^{-c_4\tau^X_{-\lfloor d/4\rfloor}}$. Putting these bounds together gives the existence of $c,\bar c>0$ such that
\begin{align}\notag
\Prob_{\bf x}^\mathrm{WMIG}(C^{X,U}<\tau_0^{D^1})\le c\left(\mathbb{E}^\mathrm{WMIG}_{\bf x}[(\tau^X_{-\lfloor d/4\rfloor})^{-1/2}]+d^{-3}\right)
&\le c\left(\mathbb{E}^\mathrm{RW}_{\lfloor d/4\rfloor}[(\tau^X_{0})^{-1/2}]+d^{-3}\right)\\
&\le \bar cd^{-1},\label{eq:WAIG4}
\end{align}
where we have used Lemma~\ref{L:LCLT} in the last inequality.
Similarly, given $T^X$, $C^{V,Y}$ is at least $\frac13 T^X$ with probability at least $e^{-c_4 T^X}$. Note also that under WMIG $\tau_0^{D^2}$ and $\tau_0$ are independent (since marking does not affect movement). Then we have the bound
\begin{align}
\Prob_{\bf x}^\mathrm{WMIG}(C^{V,Y}<\tau_0^{D^2}\wedge \tau^D)
&\le c\left(\mathbb{E}^\mathrm{RW}_{\lfloor d/4\rfloor}[(\tau^X_{0})^{-1/2}(1-p)^{\tau_0^X/3}]+d^{-3}\right)\notag\\
&\le \hat c\left(d^{-1}e^{-\hat c d\sqrt p}+d^{-2}\right),\label{eq:WAIG5}
\end{align}
for some $\hat c>0$, where we have applied Lemma~\ref{L:LCLT} in the last inequality. Combining equations \eqref{eq:WAIG1}--\eqref{eq:WAIG5} completes the proof.\qedhere
\end{proof}

We now present the proof of Proposition~\ref{P:cyclemain}.
\begin{proof}[Proof of Proposition~\ref{P:cyclemain}]Recall that $d:=d(y,x)\ge d(u,v)$.
The first step is to couple this CAB process with a WAG process denoted $({\bf X}'_t)_{t\ge0}=(X'_t,U'_t,V'_t,Y'_t)_{t\ge0}$. Let $m$ be the midpoint between $x$ and $y$, defined to satisfy 
\[
d(m,x)=\left\lfloor\frac{d}{2}\right\rfloor,\quad d(y,m)=\left\lceil\frac{d}{2}\right\rceil.
\] The coupled WAG process is initialised from\[(X'_0,U'_0,V'_0,Y'_0)={\bf x}':=(-d(m,x),-d(m,u),d(v,m),d(y,m)).\]Up until WAG terminates, the absorption times of particles in ${\bf X}'$ to the cemetery state $\Delta$ are the same as their counterparts in ${\bf X}$. If particle $X$ jumps clockwise (resp. anticlockwise), then particle $X'$ jumps down (resp. up) by 1 (similarly for other particles). Then the meeting time of particles $U$ and $V$ coincides with the time that particle $V'$ becomes a ghost.

Notice that for $X$ and $Y$ to meet before either $X$ meets $U$ or $Y$ meets $V$, at least one of $X$ and $Y$ must hit the midpoint $m$, times we denote by $\tau^X_m$ and $\tau^Y_m$. Therefore we have 
\[
\big\{M^{X,Y}\vee M^{U,V}< \tau\big\}\subseteq\big\{(\tau^X_m\wedge\tau^Y_m)\vee M^{U,V}<\tau\big\}.
\]
However, the coupling forces
\[
\big\{(\tau^X_m\wedge\tau^Y_m)\vee M^{U,V}<\tau\big\}=\big\{(\tau^{X'}_0\wedge\tau^{Y'}_0)\vee G^{V'}<\tau'\big\}\subseteq\big\{\tau^{X'}_0\wedge\tau^{Y'}_0<\tau'\big\}
\]
%
where $\tau':=\min\{\tau^{X'},\tau^{U'},\tau^{V'},\tau^{Y'},M^{X',U'},M^{V',Y'}\}$.
Hence we deduce that 
\begin{align*}
\Prob_{\bf x}^\mathrm{CAB}(M^{X,Y}\vee M^{U,V}<\tau)\le \Prob_{{\bf x}'}^\mathrm{WAG}(\tau^{X'}_0\wedge\tau^{Y'}_0<\tau').
\end{align*}
The proof is completed by combining this with Lemmas~\ref{L:WAG2WAIG} and~\ref{L:WAIGbound} (which can be applied since $n\ge8$ and so $d\ge4$, $d(m,x)\ge \frac{d}{2}-1$ and $d(y,m)\ge\frac{d}{2}-\frac12$).\end{proof}
\appendix

\begin{appendix}
\section*{}
\label{S:app}
We present here some results which are used throughout. The first is an application of Stein's method to obtain conditions (involving variance and covariance) for Gaussian convergence.
\begin{theorem}[R\"ollin \cite{rollin2008note}]\label{thm:rollin}
Suppose $W$, $W'$ are two random variables with the same distribution, $\E(W) = 0$, $\var(W) = 1$ and $\E(W'|W) = (1-\lambda)W$. Then for $\delta := \sup_{s \in \R} |\Prob(W \leq s)-\Phi(s)|$,
\begin{align}\label{eqn:rollin}
\delta \leq \frac{12}{\lambda} \sqrt{\var(\E((W'-W)^2|W))}+32\frac{A^3}{\lambda}+6\frac{A^2}{\sqrt{\lambda}},
\end{align}
where $A$ is such that $|W'-W|\leq A$, and $\Phi$ is the cumulative distribution of a standard normal distribution on $\R$.
\end{theorem}


\begin{theorem}[Local central limit theorem {\cite[Theorem 1.2.1]{LawlerIntersection}}]\label{thm:LCLT}

Let $p_n(x)$ denote the probability that a lazy simple random walk on $\Z^d$ is located in $x\in\Z^d$ after $n$ transitions starting from $0$. Then

\begin{align}
    \left |p_n(x)-2\left(\frac{d}{2\pi n}\right)^{d/2}e^{-\frac{d|x|^2}{2n}}\right| \leq \frac{c}{n^{d/2}}(|x|^{-2}\wedge n^{-1}),
\end{align}
where $|x|^2 =\sum_{i=1}^d x_i^2$, and $c$ is a universal constant.

\end{theorem}

\begin{theorem}[The hitting time theorem \cite{Hofstad2008Elementary}]\label{thm:hittingtime}
Let $(X_t)_{t\geq 0}$ be a discrete-time simple random walk on $\Z$. Denote by $T_0$ the hitting time of $0 \in \Z$. Then 
\begin{align}
    \Prob(T_0=n|X_0=k) = \frac{k}{n}\Prob(X_n = 0).
\end{align}
\end{theorem}\end{appendix}
%
%

\begin{acks}[Acknowledgments]
Nicol\'as Rivera was supported by ANID FONDECYT grant No 3210805, ANID SIA grant No 85220033, and ANID EXPLORACION grant No 13220168.

For the purposes of open access, the authors have applied a CC BY public copyright licence to any author accepted manuscript version arising from this submission.
\end{acks}
\bibliographystyle{imsart-number} 
\bibliography{ref}       


\end{document}